\documentclass[a4paper,11pt]{article}

\usepackage{makeidx}
\usepackage{amscd}
\usepackage{amsmath}
\usepackage{amssymb}
\usepackage{amsthm}

\usepackage[all]{xy}
\theoremstyle{plain}
\newtheorem{thm}{Theorem}[section]
\newtheorem{cor}[thm]{Corollary}
\newtheorem{lem}[thm]{Lemma}
\newtheorem{cla}[thm]{Claim}

\newtheorem{prop}[thm]{Proposition}
\newtheorem{conj}[thm]{Conjecture}

\newtheorem{fact}[thm]{Fact}

 \newtheorem*{lemmaA}{Lemma A}
 \newtheorem*{lemmaB}{Lemma B}
\theoremstyle{definition}

\newtheorem{rem}[thm]{Remark}
\newtheorem{exa}[thm]{Example}
\newtheorem*{division}{Division into Cases}

\newtheorem{defn-prop}[thm]{Definition-Proposition}

\newcommand{\PP}{\mathbb P}
\newcommand{\Z}{\mathbb Z}
\newcommand{\Q}{\mathbb Q}

\newcommand{\C}{\mathbb C}

\newcommand{\Ker}{\mathop{\mathrm{Ker}}\nolimits}
\newcommand{\Image}{\mathop{\mathrm{Im}}\nolimits}
\newcommand{\Coker}{\mathop{\mathrm{Coker}}\nolimits} 
\newcommand{\id}{\ensuremath{\mathop{\mathrm{id}}}}

 \newcommand{\length}{\operatorname{length}} 
 \newcommand{\Cone}{\operatorname{Cone}} 
 
\newcommand{\Pic}{\mathop{\mathrm{Pic}}\nolimits}
\newcommand{\Ext}{\mathop{\mathrm{Ext}}\nolimits}
\newcommand{\Hom}{\mathop{\mathrm{Hom}}\nolimits}

\newcommand{\RHom}{\mathop{\mb R\mathrm{Hom}}\nolimits}
\newcommand{\mcRHom}{\mathop{\mb R\mathcal{H}om}\nolimits}

\newcommand{\Lotimes}{\stackrel{\mb L}{\otimes}}
\newcommand{\Spec}{\operatorname{Spec}}
\newcommand{\GL}{\operatorname{GL}}
\newcommand{\SL}{\operatorname{SL}}

\newcommand{\Coh}{\operatorname{Coh}}

\newcommand{\Auteq}{\operatorname{Auteq}}
\newcommand{\module}{\operatorname{mod}}
\newcommand{\Aut}{\operatorname{Aut}} 
\newcommand{\ch}{\operatorname{ch}}

\newcommand{\owe}{\mathcal{O}} 
 
\newcommand{\mc}{\mathcal}
\newcommand{\mb}{\mathbb}

\newcommand{\Supp}{\ensuremath{\operatorname{Supp}}}
\renewcommand{\labelenumi}{(\roman{enumi})}

\newcommand{\FM}{\ensuremath{\operatorname{FM}}}

\newcommand{\Span}[1]{\left<#1\right>}

\newcommand{\zero}{
\setlength{\unitlength}{1ex}
\begin{picture}(2, 2)(-1, -1)
\put(0,0){\circle{2}}
\put(0,0){\makebox(0,0){\tiny$0$}}
\end{picture}
}
\newcommand{\mone}{
\setlength{\unitlength}{1ex}
\begin{picture}(2, 2)(-1, -1)
\put(0,0){\circle{2}}
\put(0,0){\makebox(0,0){\tiny $-1$}}
\end{picture}
}
\newcommand{\no}{
\setlength{\unitlength}{1ex}
\begin{picture}(2, 2)(-1, -1)
\put(0,0){\circle{2}}
\end{picture}
}

\newcommand{\cia}{\setlength{\unitlength}{1ex}
\begin{picture}(2, 2)(-1, -1)
\put(0,0){\circle{2}}
\put(0,0){\makebox(0,0){\tiny$a$}}
\end{picture}
}

\newcommand{\cib}{\setlength{\unitlength}{1ex}
\begin{picture}(2, 2)(-1, -1)
\put(0,0){\circle{2}}
\put(0,0){\makebox(0,0){\tiny$-2$}}
\end{picture}
}

\title{Autoequivalences of derived categories of elliptic surfaces with non-zero Kodaira dimension}
\author{Hokuto Uehara}
\date{}
\begin{document}
\maketitle
\begin{abstract}
We study the group of autoequivalences
 of the derived categories of coherent sheaves on 
smooth projective elliptic surfaces with non-zero Kodaira dimension.
We find a description of it when each reducible fiber is a cycle of $(-2)$-curves and non-multiple. 
\end{abstract}

\tableofcontents
\section{Introduction}
\subsection{Motivations and results}\label{subsec:motivation}
Let $X$ be a smooth projective variety over $\C$ and $D(X)$ the bounded derived category of coherent sheaves on $X$.
If $X$ and $Y$ are smooth projective varieties with equivalent derived categories, then we call
$X$ and $Y$ \emph{Fourier--Mukai partners}.
We define the set of isomorphism classes of Fourier--Mukai partner of $X$ as
$$
\FM (X):=\{Y\text{ smooth projective varieties }\mid 
D(X)\cong D(Y) \}/\cong.
$$
It is an interesting problem to determines the set 
$\FM (X)$ for a given $X$.
There are several known results in this direction.  
For example, Bondal and Orlov show that if $K_X$ or $-K_X$ is ample, then $X$ can be entirely reconstructed from $D(X)$, namely $\FM (X)=\{X\}$
\cite{BO95}.
To the contrary, there are examples of non-isomorphic varieties $X$ and $Y$ having equivalent derived categories. 
For example, in dimension $2$, if $\FM (X)\ne \{X\}$, then $X$ is a K3 surface, an abelian surface or 
a relatively minimal elliptic surface with non-zero Kodaira dimension (\cite{BM01}, \cite{Ka02}).
In dimension $3$, some results are shown by Toda \cite{To03}. 
Moreover, Orlov gives a complete answer
to this problem for abelian varieties in \cite{Or02}.   
 
It is also natural to study the isomorphism classes of autoequivalences of $D(X)$. The group 
consisting of all exact $\C$-linear autoequivalences of $D(X)$ up to isomorphism is denoted by $$
\Auteq D(X).
$$ 
We note that $\Auteq D(X)$ always contains the group
$$
A(X):=\Pic X \rtimes \Aut X\times \Z[1],
$$
generated by \emph{standard autoequivalences}, namely the functors of tensoring with invertible sheaves, pull backs along automorphisms, 
and the shift functor $[1]$.

When $K_X$ or $-K_X$ is ample, Bondal and Orlov show that 
$\Auteq D(X)\cong A(X)$. 

When $X$ is an abelian variety, Orlov determines the structure of $\Auteq D(X)$ (\cite{Or02}).
As a special case, when $X$ is an elliptic curve, the autoequivalence group is described %(cf.~\S \ref{subsec:elliptic}) 
as
\begin{equation*}%\label{eqn:elliptic}
1\to \hat{X}\rtimes\Aut X\times \Z[2] \to \Auteq D(X) \stackrel{\theta}\to \SL(2,\Z)\to 1.
\end{equation*}
Here $\theta$ is given by the action of $\Auteq D(X)$
on the even integral cohomology group 
$H^{0}(X,\Z)\oplus H^{2}(X,\Z)$, which is isomorphic to $\Z^2$.
In this case, the group $\Auteq D(X)$ contains the Fourier--Mukai transform $\Phi^{\mc{U}}_{J_X(a,b)\to X}$ 
determined by 
a universal sheaf $\mc{U}$ of the fine moduli space $J_X(a,b)$ of stable vector bundles of  the rank $a$ and the degree $b$ with $(a,b)=1$. By the work of Atiyah (\cite[Theorem 7]{At57}),  
$J_X(a,b)$ is isomorphic to $X$, and hence $\Phi^{\mc{U}}_{J_X(a,b)\to X}$ can be regarded as an autoequivalence of $D(X)$. 
One can check that $\Phi^{\mc{U}}_{J_X(a,b)\to X}$ does not belong to $A(X)$.

For the minimal resolution $X$ of $A_n$-singularities on a surface,
Ishii and the author determine the structure of $\Auteq D(X)$
 (\cite{IU05}). %, see Theorem \ref{thm:IU05}).  
It is generated by the group $A(X)$ and twist functors of the form $T_{\mathcal{O}_G(a)}$ (\cite{ST01})
%, see \S\ref{subsec:twist_functor}) 
associated with the line bundle $\mathcal{O}_G(a)$ on a $(-2)$-curve $G(\cong \PP^1)$ on $X$. 
Again,  $T_{\mathcal{O}_G(a)}$ does not belong to $A(X)$.

The case of smooth projective elliptic surfaces $\pi\colon S\to C$ with non-zero Kodaira dimension is a mixture of the last two cases.
If $S$ has a reducible fiber, each component of it is a $(-2)$-curve. Hence  
$\Auteq D(S)$ contains twist functors as in the case \cite{IU05}.
On the other hand, let us consider the fine moduli space $J_S(a,b)$ of 
 pure $1$-dimensional stable sheaves on $S$,
the general point of which represents a rank $a$, degree $b$ stable 
vector bundle supported on a smooth fiber of $\pi$. 
It often occurs that 
there is an isomorphism 
$S\cong J_S(a,b)$, and then the 
Fourier--Mukai transform $\Phi^{\mc{U}}_{J_S(a,b)\to S}$ determined by 
 a universal sheaf $\mc{U}$ on $J_S(a,b)\times S$
gives a non-standard autoequivalence.

For an object $E$ of $D(S)$, we define the fiber degree of $E$ 
\[d(E)=c_1(E)\cdot F, \]
where $F$ is a general fiber of $\pi$. Let us denote by $\lambda_{S}$  
the highest common factor of the fiber degrees of objects of $D(S)$. 
It is shown in \cite{Br98} that if $a\lambda_S$ and $b$ are coprime, the above mentioned fine moduli space $J_S(a,b)$ exists. 
We denote $J_S(b):=J_S(1,b)$.

We set
$$
B:=\Span{T_{\mathcal{O}_G(a)} \mid G \text{ is a $(-2)$-curve }}$$
and denote the congruence subgroup of $\SL (2,\Z)$ by
$$
\Gamma_0(m):=\bigl\{ \begin{pmatrix}
c& a\\
d& b   
\end{pmatrix}\in \SL (2,\Z)\bigm|d\in m\Z \bigr\}
$$
for $m\in \Z$.  

%%%
%%%
%%%

\begin{conj}\label{conj}
Let $S$ be a smooth projective elliptic surface $S$ with $\kappa (S)\ne 0$.
Then we have a short exact sequence
\begin{align*}%\label{eqn:conj_Auteq_extension}
1\to \Span{B,\otimes \mathcal{O}_S(D)\mid D\cdot F=0, \text{ $F$ is a fiber}}\rtimes \Aut S\times  
\Z[2]
\to
\Auteq D(S)& \notag\\
\stackrel{\Theta}\to 
\bigl\{ \begin{pmatrix}
c& a\\
d& b   
\end{pmatrix}\in \Gamma_0(\lambda_{S}) \bigm| J_S(b)\cong S \bigr\}
&\to 1.
\end{align*}
Here 
$\Theta$ is induced by the action of 
$\Auteq D(S)$ on the 
even degree part $H^0(F,\Z)\oplus H^2(F,\Z)\cong \Z^2$ of integral cohomology groups of  on a smooth fiber $F$.
\end{conj}

\begin{rem}
\begin{enumerate}
%\item
%The set 
%$$
%H_S:=\{ b\in (\Z/\lambda_S\Z)^* \mid J_S(b)\cong S\}
%$$ 
%forms a subgroup of $(\Z/\lambda_S\Z)^*$, and 
%the set $\FM (S)$ can be naturally identified with 
%the quotient group of $(\Z/\lambda_S\Z)^*/H_S$.
\item
As Conjecture \ref{conj} implicitly implies,  we  can actually see 
in the proof of Theorem \ref{thm:AeV_extension}
that every element of $\Auteq D(S)$ induces an autoequivalence of a smooth fiber $F$.
If $\kappa(S)= 0$, this is false. See in Example \ref{exa:twist} (iii).
\item 
The quotient group  
$\Gamma_0(\lambda_S)/ \Image \Theta$ 
is naturally identified with the set of Fourier--Mukai partners $\FM (S)$ of $S$.
See Remark \ref{rem:image_Theta}.
\item If $\pi$ has a section, we know that 
$\Image \Theta \cong \SL (2,\Z)$. See Remark \ref{rem:image_Theta}.
\item
By \cite[Proposition 4.18]{IU05}, we have
$$
B\cap\Span{\otimes \mathcal{O}_S(D)\mid D\cdot F=0, \text{ $F$ is a fiber}}
=\Span{\otimes\mathcal{O}_S(G)\mid \text{ $G$ is a $(-2)$-curve}}.$$
\end{enumerate}
\end{rem}

The following is the main result in this article.

%%%
%%%
%%% 

\begin{thm}[=Theorem \ref{thm:typeI_n}]\label{thm:main}
Suppose that each reducible fiber  on the elliptic surface $S$ is non-multiple, and forms a cycle of $(-2)$-curves,
 i.e. of type $\mathrm{I}_n$ for some $n>1$. 
Then Conjecture \ref{conj} is true.
\end{thm}

In the case that $\pi$ has only irreducible fibers, and 
$\pi$ has a section, 
Conjecture \ref{conj} is essentially shown in \cite{LST13}.\footnote{They only consider autoequivalences $\Phi$ with 
$\Phi\in\Ker\delta$. See the definition of $\delta$ in Remark \ref{rem:image_delta}.
On the other hand, their result is valid without any restrictions on the base space of the fibration.}

\subsection{Outline of the proof of Theorem \ref{thm:main}}%  and the construction of this article}
Let $S$ be a projective elliptic surface with $\kappa(S)\ne 0$, and 
$Z$ be the union of all reducible fibers of the elliptic fibration $\pi\colon S\to C$, and 
$U$ be the complement of $Z$ in $S$. 
We introduce a group homomorphism 
$$
\iota_U\colon \Auteq D(S)\to \Auteq D(U)
$$
in Proposition \ref{prop:U} and denote $\Image \iota_U$ by $\Auteq ^{\dagger} D(U)$.
We classify all elements of $\Auteq ^{\dagger} D(U)$ by Proposition \ref{prop:noreducible} 
and determine the structure of $\Auteq ^{\dagger} D(U)$ in Theorem \ref{thm:AeV_extension}.

Assume furthermore that all reducible fibers of $\pi$ are of type $\mathrm{I}_n$ ($n>1$) as in Theorem \ref{thm:main}.
Then we can show that 
\begin{equation}\label{eq:B=kerU}
B=\Ker\iota_U
\end{equation}
by using
\begin{prop}[=Proposition \ref{proposition:step -2 of A_n}]\label{proposition:step -2 of A_n intro}
%We use the notation in Theorem \ref{thm:main}, and 
Take a connected component $Z_0$ of $Z$, that means a fiber of $\pi$.
Let us consider the irreducible decomposition $Z_0=C_1\cup\cdots\cup C_n$, where each $C_i$ is a $(-2)$-curve. 
Suppose that we are given an autoequivalence $\Phi$ of $D_{Z_0}(S)$ preserving
the cohomology class $[\mc{O}_x] \in H^4(S,\Q)$ for some point $x\in Z_0$.
Then, there are integers $a$, $b$ $(1\le b\le n)$ and $i$,
and there is an autoequivalence 
$$
\Psi\in 
%B_0:=
\Span{T_{\mc{O}_G(a)}\mid G \text{ is a $(-2)$-curves contained in $Z_0$}}
$$
such that
$$
\Psi \circ \Phi(\owe_{C_1}) \cong \owe_{C_{b}}(a)[i]
$$
and
$$
\Psi \circ \Phi(\owe_{C_1}(-1)) \cong \owe_{C_{b}}(a-1)[i].
$$
In particular, 
for any point $x\in C_1$, we can find a point $y\in C_b$ 
with $\Psi\circ\Phi (\mc{O}_x)\cong \mc{O}_y[i]$.  
\end{prop}
\noindent
Proposition \ref{proposition:step -2 of A_n intro} is proved in
\S \ref{sec:-2} using techniques developed in \cite{IU05}.
Then we can deduce Theorem \ref{thm:main} from the equation \eqref{eq:B=kerU} and the description 
of $\Auteq ^{\dagger} D(U)$ obtained in  Theorem \ref{thm:AeV_extension}.

The construction of this article is as follows.
In \S \ref{sec:preliminaries} we show several preliminary results and give definitions needed afterwards.
In \S \ref{sec:autoeq_non-zero_Kodaira}, we study the structure of $\Auteq ^{\dagger}D(U)$ intensively.
In \S \ref{sec:I_n}, we reduce the proof of Theorem \ref{thm:main}
to showing the equation \eqref{eq:B=kerU}, and furthermore we reduce the proof of \eqref{eq:B=kerU}
to showing Proposition \ref{proposition:step -2 of A_n intro}. In \S \ref{sec:I_n},
we also show several lemmas used in \S \ref{sec:-1}
and \S\ref{sec:-2}.
In \S \ref{sec:-1}, we show Proposition \ref{proposition:step -1 of A_n}, which is the first step 
in the proof of Proposition \ref{proposition:step -2 of A_n intro}.
In \S \ref{sec:-2}, we prove Proposition \ref{proposition:step -2 of A_n intro}.
Finally in \S \ref{sec:examples}, we treat an example of elliptic surfaces satisfying the assumption in Theorem \ref{thm:main}. In the example, we can determine the set $\FM(S)$, and also know when $J_S(b)\cong S$ holds.
This information gives us a better description of $\Image\Theta$ (see Conjecture \ref{conj}) in the example.

\subsection{Notation and conventions}\label{subsec:notation_convention}
All varieties will be defined over $\C$. 
A \emph{point} on a variety will always mean a closed point.
%For a closed subscheme $Z$ with a compact support, $[Z]$ denotes its cohomology class. 
By an \emph{elliptic surface}, we will
always mean a smooth surface $S$ together with a smooth 
curve $C$ and a relatively minimal projective morphism $\pi\colon S\to C$ whose general fiber is an elliptic curve. Here  a \emph{relatively minimal morphism} means a morphism whose fibers contains no $(-1)$-curves.

For two elliptic surfaces 
$\pi\colon S\to C$ and $\pi'\colon S'\to C$,
an isomorphism $\varphi \colon S\to S'$ satisfying 
$\pi=\pi'\circ \varphi$ is called an \emph{isomorphism over} $C$.

For an elliptic curve $E$ and some positive integer $m$, we denote  the set of points of order $m$ by ${}_mE$.  
Furthermore, we denote the dual elliptic curve  
%, namely $\hat{E}$ is just the group scheme $\Pic ^0 E$ of line bundles on $E$ of the degree $0$.
by $\hat{E}:=\Pic ^0 E$. 
 
$D(X)$ denotes the bounded derived category of coherent sheaves on an algebraic variety $X$. For a closed subset $Z$ of $X$, we denote the full subcategory of $D(X)$ consisting of objects supported on $Z$ by $D_Z(X)$.
Here, the support of an object of $D(X)$ is, by definition, the union of the set-theoretic supports of its cohomology sheaves.  

An object $\alpha$ in $D(X)$ is said to be \emph{simple} (respectively \emph{rigid}) if 
$$
\Hom_{D(X)}(\alpha,\alpha)\cong \C \mbox{ (respectively }\Hom^1_{D(X)}(\alpha,\alpha)\cong 0).
$$ 
Given a closed embedding of schemes  $i\colon Z\hookrightarrow X$,
we often denote the derived pull back 
$\mb Li^*\alpha$ simply by $\alpha|_Z$. 

$\Auteq \mathcal{T}$ denotes the group of isomorphism classes of $\C$-linear exact autoequivalences of a $\C$-linear triangulated category $\mathcal{T}$.

\subsection{Acknowledgements} 
The author is supported by the Grants-in-Aid 
for Scientific Research (No.23340011). 
This paper was written during his staying 
at the Max-Planck Institute for Mathematics, 
in the period from April to September 2014.
The author appreciates the hospitality. 
He also thanks Yukinobu Toda for pointing out some mistakes
in the first draft.
  
%%%%%%%%%%%%%%%%%%%%%%%%%%%%%%%%%%%%%%%%%%%%%%%%%%%%%%%%%%%%%%%%%%%%%%%%%%%%%%%%%%%%%%%%%%%%%%
%%%%%%%%%%%%%%%%%%%%%%%%%%%%%%%%%%%%%%%%%%%%%%%%%%%%%%%%%%%%%%%%%%%%%%%%%%%%%%%%%%%%%%%%%%%%%%
%%%%%%%%%%%%%%%%%%%%%%%%%%%%%%%%%%%%%%%%%%%%%%%%%%%%%%%%%%%%%%%%%%%%%%%%%%%%%%%%%%%%%%%%%%%%%%

\section{Preliminaries}\label{sec:preliminaries}

%%%%%%%%%%%%%%%%%%%%%%%%%%%%%%%%%%%%%%%%%%%%%%%%%%%%%%%%%%%%%%%%%%%%%%%%%%%%%%%%%%%%%%%%%%%%%%
\subsection{General results for Fourier--Mukai transforms}

Let $X$ and $Y$ be smooth projective varieties.
We call $Y$ 
\emph{a Fourier--Mukai partner 
of $X$} 
if $D(X)$ is $\C$-linear triangulated equivalent to $D(Y)$.
We denote by $\FM (X)$ the set of isomorphism classes of Fourier--Mukai partners of $X$.

For an
object $\mathcal{P}\in D(X\times Y)$, we define an exact functor $\Phi^{\mathcal{P}}$, called an \emph{integral functor}, to be
$$
\Phi^{\mathcal{P}}:= 
\mathbb{R}p_{Y*}(\mathcal{P}\Lotimes p^{*}_X(-))\colon D(X)\to D(Y),
$$
where we denote the projections by $p_X\colon X\times Y\to X$ and $p_Y\colon X\times Y\to Y$.
We also sometimes write $\Phi^{\mathcal{P}}$ as $\Phi^{\mathcal{P}}_{X\to Y}$ to emphasize that it is a functor from $D(X)$ to $D(Y)$.

Next suppose that $X$ and $Y$ are not necessarily projective.
Then, in general, $\mathbb{R}p_{Y*}$ is not well-defined as a functor $D(X\times Y)\to D(Y)$ since $p_Y$ is not projective.
Instead, suppose that there are projective morphisms 
$X\to C$ and $Y\to C$ over a smooth variety $C$, and  let $\mathcal{P}$  be  a perfect complex in $D(X\times _CY)$. Then we can also define 
the integral functor in this case, by replacing 
the projections $p_Y$ and  $p_X$ with $p_Y\colon X\times_C Y\to Y$ and  $p_X\colon X\times_C Y\to X$ respectively (note that  we use the same notation for both kinds of projections). 
If we want to emphasize that we are in this situation, 
$\Phi$ is called a \emph{relative integral functor over $C$}. 
Later, we use relative integral transforms 
in the case of elliptic surfaces over a non-projective base $C$.    
 
By the result of Orlov (\cite{Or97}),
for smooth projective varieties $X$ and $Y$,
and for a fully faithful functor $\Phi\colon D(X)\to D(Y)$,
there is an object $\mathcal{P}\in D(X\times Y)$, unique up to isomorphism, such that 
$$
\Phi\cong \Phi^{\mathcal{P}}.
$$
If an integral functor (over $C$) is an equivalence,
% and its quasi-inverse is also given as an  integral functor (over $C$)
it is called a \emph{Fourier--Mukai transform (over $C$)}.

The left adjoint to  an integral functor $\Phi^{\mathcal{P}}$ over $C$ is given 
by the integral functor $\Phi^{\mc{Q}}$ over $C$ where 
$$
\mc{Q}:=\mcRHom _{X\times_C Y}(\mc{P},\mc{O}_{X\times_C Y})\Lotimes p_X^*\omega_{X/C}[\dim X-\dim C]
$$
(see the proof of \cite[Proposition 5.9]{Hu06}). In particular, if  $\Phi^{\mathcal{P}}$ is an equivalence,
its quasi-inverse is given by $\Phi^{\mc{Q}}$.

We can also see that the composition of integral functors over $C$ is again an integral functor over $C$ (cf.~\cite[Proposition 5.10]{Hu06}).

%%%
%%%
%%%

\begin{lem}\label{lem:FMfamily}
Let $\Phi\colon D(X)\to D(Y)$ be a Fourier--Mukai transform over  a smooth variety $C$ between smooth varieties $X, Y$, projective over $C$.
Then the set of points $x\in X$ for which the object
$\Phi (\mathcal{O}_x)$ is a sheaf forms an (possibly empty) open subset of $X$.
\end{lem}

\begin{proof}
This is a special case of \cite[Proposition 2.4]{BM01}. See also
the proof of \cite[Lemma 2.5]{BM01}.
\end{proof}

The following is well-known.

%%%
%%%
%%%

\begin{lem}\label{lem:birational}
Let $\Phi\colon D(X)\to D(Y)$ be a Fourier--Mukai transform over a smooth variety $C$ between smooth varieties $X, Y$, projective over $C$.
 Assume that 
it satisfies that $\Phi (\mathcal{O}_x)$ is a shift of a sheaf 
supported  on a finite subset of $Y$ 
for all points $x\in X$.
Then we have 
$$\Phi \cong \phi _*\circ((-)\otimes \mathcal{L})[n]$$  
for a line bundle $\mathcal{L}$ on $X$,
an isomorphism $\phi\colon X\to Y$ over $C$ and some integer $n$.
\end{lem}

\begin{proof}
%We know that $\Hom_{D(Y)}(\Phi(\mathcal{O}_x),\Phi(\mathcal{O}_x))=\C$, 
By the assumptions,  $\Phi(\mathcal{O}_x)$ 
satisfies the condition of \cite[Lemma 4.5]{Hu06}.
Hence, 
$\Phi(\mathcal{O}_x)\cong \mathcal{O}_y[n]$ for some $y\in Y$
and $n\in \Z$.
Note that the integer $n$ does not depend on the choice of a point $x$ by Lemma \ref{lem:FMfamily}.
Then apply  \cite[\S3.3]{BM98} (or \cite[Corollary 5.23]{Hu06}) to get the conclusion.
\end{proof}

The following Lemma is useful.

%%%
%%%
%%%

\begin{lem}\label{lem:BM02}
Let $X$ be a smooth variety.
For an object $E\in D(X)$ with a compact support,
$\RHom _{D(X)} (E,\mathcal{O}_x)\ne 0$ if and only if 
a point $x$ is contained in $\Supp E$.
Moreover, these conditions are
 also equivalent to
$\RHom _{D(X)} (\mathcal{O}_x,E)\ne 0$.
\end{lem}

\begin{proof}
The first statement is just \cite[Lemma 5.3]{BM02}.
The second follows from the Grothendieck--Serre duality and the first.
\end{proof}

%%%
%%%
%%%

\begin{lem}\label{lem:Z}
Let $X$ and $Y$ be smooth projective varieties together with 
closed subsets $Z \subset X$ and $W\subset Y$. 
Suppose that a Fourier--Mukai transform  
$\Phi =\Phi_{X\to Y}$ 
and its quasi-inverse $\Psi$ satisfy that 
$$
\Supp \Phi (\mc{O}_x)\subset W \text{ and } 
\Supp \Psi (\mc{O}_{y})\subset Z.
$$
for any point $x\in Z,y\in W$
Then  
$\Phi$ restricts to a Fourier--Mukai transform from $D_Z(X)$ to $D_{W}(Y)$. 
\end{lem}

\begin{proof}
We repeatedly use Lemma \ref{lem:BM02}.
Take a point $x\in Z$ and $y\in Y\backslash W$. Then
$$
\RHom _{D(X)}(\Psi (\mathcal{O}_{y}),\mathcal{O}_x)
=\RHom _{D(Y)}(\mathcal{O}_{y},\Phi(\mathcal{O}_x))
$$
vanishes by the assumption.
This means that  
$\Supp\Psi(\mathcal{O}_{y})\cap Z=\emptyset$, and
thus for an object $E\in D_Z(X)$, we have
$$
\RHom _{D(Y)}(\Phi(E),\mathcal{O}_y)=
\RHom _{D(X)}(E,\Psi(\mathcal{O}_y))
=0,
$$
which implies that $\Phi(E)\in D_{W}(Y)$.
Here, the last equality follows from \cite[Lemma 3.9]{Hu06}.
%(or see also \cite[Example 3.2]{Br99}).
In a similar way, we can prove that 
$\Psi(F)\in D_Z(S)$ for any objects $F\in D_{W}(Y)$.
Therefore, we obtain the conclusion.
\end{proof}

The following is also well-known.

%%%
%%%
%%%

\begin{lem}[cf.~Proposition 2.15 in \cite{HLS09}]\label{lem:base_change}
Let $\pi\colon X\to C$ and $\pi'\colon Y\to C$ be flat projective morphisms between smooth varieties, 
and take a  point $c$ of $C$.
Suppose that $X_c$ and $Y_c$ are the fibers of $\pi$ and $\pi'$ respectively over the point $c$,\footnote{Although we do not assume that $X_c$ and $Y_c$ are smooth, the perfectness of $\mc{P}|_{X_c\times Y_c}$ assures that $\Psi$ defines a functor from $D(X_c)$ to $D(Y_c)$.} and that
we are given an integral functor  
$\Phi =\Phi_{X\to Y}^{\mc{P}}$ over $C$. 

Let us consider the integral functor 
$\Psi =\Phi_{X_c\to Y_c}^{\mc{P}|_{X_c\times Y_c}}$, and  
denote  the inclusions by  $k\colon X_c\hookrightarrow  X$ and 
$k'\colon Y_c\hookrightarrow  Y$.
Then we have the following.
\begin{enumerate}
\item
$\Phi$ and $\Psi$ satisfies
$
k'_*\circ  \Psi \cong  \Phi \circ k_*.
$
\item
Assume furthermore that $\Phi$ is an Fourier--Mukai transform.
Then so is $\Psi$.
\end{enumerate}
\end{lem}

\begin{proof}
Assertion (i) directly follows from the projection formula and the flat base change formula (cf.~\cite[Pages 83, 85]{Hu06}).
For (ii), suppose that $X=Y$ and $\Phi\cong\id_X$. Then obviously $\Psi \cong\id_{X_c}$ also holds.
Apply this argument to the functor 
$\Phi\circ \Phi^{-1}$ to get the result.
\end{proof}

%%%%%%%%%%%%%%%%%%%%%%%%%%%%%%%%%%%%%%%%%%%%%%%%%%%%%%%%%%%%%%%%%%%%%%%%%%%%%%%%%%%%%%%%%%%%%%%%%%%

\subsection{The Euler form on surfaces}\label{subsec:euler_form}
Let $X$ be a smooth quasi-projective variety. For objects $\alpha,\beta\in D(X)$ with compact supports, we define the Euler form as 
$$
\chi (\alpha,\beta):=\sum_i (-1)^i\dim\Hom^i_{D(X)}(\alpha,\beta).
$$

In the surface case, the Riemann-Roch theorem yields
\begin{align*}
\chi(\alpha,\beta)&=r(\alpha)\ch_2(\beta)-c_1(\alpha)\cdot c_1(\beta)+r(\beta)\ch_2(\alpha)\\
                         &+\frac{1}{2}(r(\beta)c_1(\alpha)-r(\alpha)c_1(\beta))\cdot c_1(\omega_X)+r(\alpha)r(\beta)\chi(\mc{O}_X).
\end{align*}
In particular, if $r(\alpha)=r(\beta)=0$, we have
\begin{equation}\label{eqn:euler}
\chi(\alpha,\beta)=-c_1(\alpha)\cdot c_1(\beta).
\end{equation}
As its application, for a $(-2)$-curve $G$ on a smooth surface $X$,
we can compute 
\begin{align}\label{ali:ext^1}
\dim\Ext_X^1(\mc{O}_G(a),\mc{O}_G(b))=
\begin{cases}
b-a-1 \quad&\mbox{ if } b-a>1\\
0               &\mbox{ if } |b-a|\le 1\\
a-b-1 &\mbox{ if } a-b>1.
\end{cases}
\end{align}

%%%%%%%%%%%%%%%%%%%%%%%%%%%%%%%%%%%%%%%%%%%%%%%%%%%%%%%%%%%%%%%%%%%%%%%%%%%%%%%%%%%%%%%%%%%%%%
\subsection{Twist functors}\label{subsec:twist_functor}

We introduce an important class of examples of autoequivalences. 

%%%
\begin{defn-prop}[\cite{ST01}]\label{def-prop:twist_functor}
Let $X$ be a smooth variety, or rather a complex manifold.
\begin{enumerate}
\item
We say that an object $\alpha \in D(X)$ with a compact support is \emph{spherical} if 
we have $\alpha \otimes \omega_{X} \cong \alpha$ and
$$
\Hom^{k}_{D(X)}(\alpha,\alpha)\cong\begin{cases}  0 & k\ne 0,\dim X\\
                                                \C & k=0,\dim X. 
\end{cases}
$$
\item
Let $\alpha\in D(X)$ be a spherical object. 
We consider the mapping cone 
$$
\mc{C}=\Cone(\pi_1^*\alpha^\vee\Lotimes \pi_2^*\alpha \to \mc{O}_{\Delta})
$$
of the natural evaluation $\pi_1^*\alpha^\vee\Lotimes \pi_2^*\alpha \to \mc{O}_{\Delta}$,
where $\Delta\subset X\times X$ is the diagonal, and $\pi_i$ is the projection of
$X\times X$ to the $i$-th factor. Then the integral functor 
$T_{\alpha}:=\Phi^{\mc{C}}_{X\to X}$ defines an autoequivalence of the bounded derived category $D_{\Coh (X)}(\mc{O}_X-\module)$
of $\mc{O}_X$-modules with coherent cohomology on $X$,\footnote{For an algebraic variety $X$, or a compact complex analytic space $X$ of dimension $2$, it is known that 
 $D_{\Coh (X)}(\mc{O}_X-\module)$ is equivalent to
 $D(X)$. See \cite[Corollary 3.4, Proposition 3.5]{Hu06} and \cite[Corollary 5.2.2]{BV03}.} 
called the \emph{twist functor} along the spherical object $\alpha$. 
\end{enumerate}
\end{defn-prop}

\begin{rem}\label{rem:twist}
Let  $\alpha\in D(X)$ be a spherical object.  Then, by the definition, 
for every $\beta\in D(X)$, we have an exact triangle
\begin{equation}\label{eqn:twist_triangle}
\RHom_{D(X)} (\alpha, \beta)\otimes_{\C}\alpha\to \beta\to T_{\alpha}(\beta). 
\end{equation}
Suppose that $\Supp\beta\cap \Supp\alpha=\emptyset$. 
Then since $\RHom (\alpha, \beta)=0$, we have $T_{\alpha}(\beta)\cong \beta$. We use this remark later.
Furthermore, in the Grothendieck group $K(X)$, we have
\begin{equation}\label{eqn:twist_Grothendieck}
[T_{\alpha}(\beta)]=[\beta]-\chi(\alpha, \beta)[\alpha] .
\end{equation}
\end{rem}

\begin{exa}\label{exa:twist}
\begin{enumerate}
\item
Let $S$ be a smooth surface,
and let $G$ be a $(-2)$-curve.
Then, for every integer $a$ and any point $x\in S$,  
we can see 
$$
\chi(\mc{O}_G(a),\mc{O}_G(a))=2 \mbox{ and }
\chi(\mc{O}_G(a),\mc{O}_x)=0
$$
by the equality \eqref{eqn:euler}.  
Hence, the object $\mc{O}_G(a)\in D(S)$ is spherical, and
it follows from the equality \eqref{eqn:twist_Grothendieck} that 
$[T_{\mc{O}_G(a)}(\mc{O}_x)]=[\mc{O}_x]$. 

By using \eqref{eqn:twist_triangle}, 
we can also see that 
$$
H^{-1}(T_{\mc{O}_G}(\mc{O}_G(2)))=\mc{O}_G(-1)^{\oplus 2} \mbox{ and } 
H^{0}(T_{\mc{O}_G}(\mc{O}_G(2)))=\mc{O}_G,
$$
and hence $T_{\mc{O}_G}\notin A(S)$.
\item
Let $\alpha$ be a simple coherent sheaf on an elliptic curve $E$, for example a line bundle or
the structure sheaf $\mathcal{O}_x$ of a point $x\in E$. Then, $\alpha$ is spherical. 
Usually twist functors are not standard autoequivalences, but 
we can see that in this case  
$T_{\mathcal{O}_x}\cong\otimes \mathcal{O}_E(x)\in A(E)$ (cf.~\cite[Example 8.10]{Hu06}).
\item
Let $X$ be a K3 surface. Then, the structure sheaf $\mc{O}_X$ of $X$ is spherical.
We can see that $\Supp T_{\mc{O}_X}(\mc{O}_x)=2$ by using the triangle \eqref{eqn:twist_triangle}. 
To the contrary,
we will see in Claim \ref{cla:FM_support} that, for an elliptic surface $S$ with non-zero Kodaira dimension,
any autoequivelnce $\Phi\in\Auteq D(S)$ satisfies $\Supp \Phi(\mc{O}_x)\le 1$.   
\end{enumerate}
\end{exa}

%%%
%%%
%%%

\begin{thm}[Theorem 1.3 in \cite{IU05}, Appendix A in \cite{IUU10}]\label{thm:IU05}
Let $X$ be a minimal resolution of the $A_n$-singularity $\Spec \C[[x, y, z]]/(x^2+y^2+z^{n+1})$.
Define 
$$
B:=\Span{T_{\mathcal{O}_G(a)} \mid G \text{ is a $(-2)$-curve }}$$
and denote by $Z$ the exceptional set of the resolution.
Then we have
$$
\Auteq D_Z(X)=(\Span{B, \Pic X} \rtimes \Aut X) \times \Z.
$$
\end{thm}

We will use ideas of the proof of Theorem \ref{thm:IU05} to prove our main result, namely Theorem \ref{thm:typeI_n}.

%%%%
\subsection{Autoequivalences of elliptic curves}\label{subsec:elliptic}
Let $E$ be an elliptic curve.
To a given $\Phi=\Phi^{\mathcal{P}} \in \Auteq D(E)$, 
we associate  a group automorphism 
$\rho(\Phi)$ of $H^*(E,\Z)\cong \Z^4$ 
given by
$$
\rho(\Phi^{\mathcal{P}})(-):=p_{2*}(\ch (\mathcal{P})\cdot p_1^*(-))
$$
(cf.~\cite[Corollary 9.43]{Hu06}). 
This gives a group homomorphism 
$$
\rho\colon \Auteq D(E) \to \GL (H^*(E,\Z)),
$$
and it is known that $\rho$ preserves the parity, i.e.
it decomposes as 
$
\rho=\eta\oplus \theta,
$
where
$\eta(\Phi)\in \GL(H^\mathrm{odd}(E,\Z))$ and 
$\theta(\Phi)\in \GL(H^\mathrm{ev}(E,\Z))$.

Because $\theta (\Phi)\in \GL(2,\Z)$ preserves the Euler form $\chi(\makebox[2mm]{\hrulefill},\makebox[2mm]{\hrulefill})$, we can see that 
$\theta (\Phi)$ actually gives an element of $\SL(2,\Z)$.
Take the classes $\ch (\mathcal{O}_E)$ and $\ch (\mathcal{O}_x)$ 
for some point $x$ as 
a basis of $H^\mathrm{ev}(E,\Z)\cong \Z^2$. In terms of this basis,
$T_{\mathcal{O}_E}$, 
$T_{\mathcal{O}_x}(\cong\makebox[2mm]{\hrulefill}\otimes \mathcal{O}_E(x))$ and 
$\Phi^\mathcal{U}$ are given by  
\begin{equation*}
\begin{pmatrix}
1& -1\\
0& 1   
\end{pmatrix},\quad
\begin{pmatrix}
1& 0\\
1& 1   
\end{pmatrix},\quad
\begin{pmatrix}
0& 1\\
-1& 0   
\end{pmatrix}
\end{equation*}
respectively, where $\mathcal{U}$ is the normalized Poincare bundle on 
$E\times E$. Here, note that every elliptic curve is principally polarized, 
and hence we can identify $E$ with $\hat E$. 
Two of these elements actually generate the group 
$\SL (2,\Z)$, and 
therefore the map 
$$
\theta\colon \Auteq D(E)\to \SL(2,\Z)
$$
is surjective. One can compute the kernel of $\theta$ to get   
 a short exact sequence of groups
\begin{equation*}%\label{eqn:elliptic}
1\to  \hat{E}\rtimes\Aut E\times \Z[2] \to \Auteq D(E)\to \SL(2,\Z)\to 1.
\end{equation*}
%

%%%%%%%%%%%%%%%%%%%%%%%%%%%%%%%%%%%%%%%%%%%%%%%%%%%%%%%%%%%%%%%%%%%%%%%%%%%%%%
%%%%%%%%%%%%%%%%%%%%%%%%%%%%%%%%%%%%%%%%%%%%%%%%%%%%%%%%%%%%%%%%%%%%%%%%%%%%%%

\subsection{Automorphisms of elliptic surfaces}\label{subsec:automorphism_eliiptic_surface}
Let $\pi\colon S\to C$ and $\pi'\colon S'\to C'$
be projective elliptic surfaces, and suppose that  each of $S$ and $S'$ has
a unique elliptic fibration.  
Then every isomorphism $\varphi\colon S\to S'$
induces an isomorphism $C\to C'$.
In the cases that we consider (namely that $S'=J_S(a,b)$; see \S\ref{subsec:bridgeland})
there is a natural identification between $C$ and $C'$. 
Hence, the induced isomorphism is naturally 
regarded as an automorphism of $C$. We denote it by $\varphi _C$.
In other words, $\varphi_C$ satisfies $\pi'\circ\varphi=\varphi_C\circ\pi$. 
%by the identification between $C$ and $C'$.
We define
$$
\Aut_SC:=\{\varphi _C \in \Aut C \mid \varphi \in \Aut S\}
$$
and 
$$
\Aut_{C}S:=\{\varphi \in \Aut S \mid \varphi_C={\id}_C \}.
$$
Consequently, we have a short exact sequence
$$
1\to \Aut_{C}S \to \Aut S \to \Aut_SC\to 1.
$$

An elliptic surface $S$ with $\kappa(S)\ne 0$ provides an example of a surface admitting a unique elliptic fibration: 
The canonical bundle formula of elliptic surfaces implies that any elliptic fibration on $S$ is defined 
by the linear system $|rK_{S}|$ with some nonzero rational number $r$. Therefore,
$S$ has a unique elliptic fibration structure.
%%%%%%%%%%%%%%%%%%%%%%%%%%%%%%%%%%%%%%%%%%%%%%%%%%%%%%%%%%%%%%%%%%%%%%%%%%%%%%
%%%%%%%%%%%%%%%%%%%%%%%%%%%%%%%%%%%%%%%%%%%%%%%%%%%%%%%%%%%%%%%%%%%%%%%%%%%%%%

\subsection{Fourier--Mukai transforms on elliptic surfaces}\label{subsec:bridgeland}

 Bridgeland, Maciocia and Kawamata show (\cite{BM01}, \cite{Ka02}) 
that for a smooth projective surface $S$,
if $S$ has a non-trivial Fourier--Mukai partner $T$, that is $|\FM (S)|\ne 1$,
then both of $S$ and $T$ are abelian varieties, K3 surfaces or minimal elliptic surfaces with 
non-zero Kodaira dimensions. 

We consider the last case in more detail. 
Let $\pi:S\to C$ be an elliptic surface. 
The results referred to in \S \ref{subsec:bridgeland} are originally stated under the assumption that $S$ is projective, but 
some of them still hold true without the projectivity of $S$. 
For our purpose, it is sometimes important to consider non-projective elliptic surfaces, hence
we do not assume that $S$ is projective unless specified otherwise.

For an object $E$ of $D(S)$, we define the fiber degree of $E$ as
\[d(E)=c_1(E)\cdot F, \]
where $F$ is a general fiber of $\pi$. 
Let us denote by $r(E)$ the rank of $E$ and by $\lambda_{S}$  
the highest common factor of the fiber degrees of objects of $D(S)$. 
Equivalently,
$\lambda_{S}$ is the smallest number $d$ such that there is a 
holomorphic $d$-section of $\pi$. 
Consider integers $a$ and $b$ with $a>0$ and $b$ coprime to $a\lambda_{S}$. 
By \cite{Br98}, there exists a smooth,
$2$-dimensional component $J_S (a,b)$ of the moduli space of pure dimension 
one stable sheaves on $S$,
the general point of which represents a rank $a$, degree $b$ stable 
vector bundle supported on a smooth fiber of $\pi$. 
There is a natural morphism $J_S (a,b)\to C$, taking a point representing
 a sheaf supported on the
fiber $\pi ^{-1}(x)$ of $S$ to the point $x$. This morphism is a minimal 
elliptic fibration (see \cite{Br98}).
Put $J_S(b):=J_S(1,b)$. 
Obviously, $J_S(0)\cong J(S)$, the Jacobian surface associated to $S$, 
and $J_S(1)\cong S$. 
As shown in \cite[Lemma 4.2]{BM01}, there is also an isomorphism
\begin{equation}\label{eqn:J(a,b)=J(b)}
J_S(a,b)\cong J_S(b)
\end{equation} 
over $C$. 

%%%
%%%
%%%

\begin{thm}[Theorem 5.3 in \cite{Br98}]\label{thm:SLFM}
Let $\pi\colon S\to C$ be an elliptic surface and take an element 
$$M=\begin{pmatrix}
c& a\\
d& b   
\end{pmatrix}
\in \SL (2,\Z)
$$
such that $\lambda_{S}$ divides $d$ and $a>0$. Then there exists a
universal sheaf $\mathcal{U}$ on $J_S(a,b)\times S$, flat over both factors,
such that for any point $(x,y)\in J_S(a,b)\times S$, $\mathcal{U}|_{x\times S}$ has Chern class $(0,af,-b)$ on $S$ and 
$\mathcal{U}|_{J_S(a,b)\times y}$ has Chern class $(0,af,-c)$ on $J_S(a,b)$.
 The resulting functor $\Phi ^{\mathcal{U}}_{J_S(a,b)\to S}$ is an equivalence and satisfies
\begin{equation}\label{eqn:matrix}
\begin{pmatrix}
r(\Phi(E))\\
d(\Phi(E))  
\end{pmatrix}
=
M
\begin{pmatrix}
r(E)\\
d(E)   
\end{pmatrix}
\end{equation} 
for all objects $E\in D(J_S(a,b))$ 
\end{thm}

\begin{rem}\label{rem:d_divides_lambda}
For integers $a>0$ and $b$ with $b$ coprime to $a\lambda_{S}$, let us consider the Fourier--Mukai transform 
$\Phi=\Phi ^{\mathcal{U}}_{J_S(a,b)\to S}$.
Take a smooth fiber $F$ of $\pi$ over a point $c\in C$,
and denote by $F'$ 
the smooth fiber of the morphism $J_S(a,b)\to C$ over the point $c$.
It turns out that $F'$ is isomorphic to $J_F(a,b)$, and hence 
$F'$ is isomorphic to $F$ by \cite[Theorem 7]{At57}. 
The integral transform defined by the kernel 
$\mc{U}|_{F\times F'}\in D(F\times F')$
induces an equivalence between $D(F)$ and $D(F')$ by Lemma \ref{lem:base_change}.
Fixing an isomorphism $F\cong F'$, we regard the equivalence 
$\Phi^{\mc{U}|_{F\times F'}}$ as an autoequivalence of $D(F)$.

Note that 
$$
M:=\theta(\Phi^{\mc{U}|_{F\times F'}}):=
\begin{pmatrix}
c& a\\
d& b   
\end{pmatrix}
$$ 
satisfies \eqref{eqn:matrix}
(see \S \ref{subsec:elliptic} for the definition of $\theta$).
We see that  $\lambda_S$ divides $d=d(\Phi (\mc{O}_{J_S(a,b)}))$.
We will use this fact in \S \ref{subsec:non-reducible:structure}.
\end{rem}

Theorem \ref{thm:SLFM} implies that 
$J_S(b)(\cong J_S(a,b))$ is a Fourier--Mukai partner of $S$ when $(b,\lambda_S)=1$. Actually, the converse is also true for 
projective elliptic surfaces $S$ with non-zero Kodaira dimension:

%%%%%%%%%
%%%%%%%%%
%%%%%%%%%

\begin{thm}[Proposition 4.4 in \cite{BM01}]\label{BMelliptic}
Let $\pi :S\to C$ be a projective elliptic surface 
and $S'$ a smooth projective variety.
Assume that the Kodaira dimension $\kappa (S)$ is non-zero.
Then the following are equivalent.
\renewcommand{\labelenumi}{(\roman{enumi})}
\begin{enumerate}
\item 
$S'$ is a Fourier--Mukai partner of $S$. 
\item 
$S'$ is isomorphic to $J_S(b)$ for some integer $b$ with $(b,\lambda _{S})=1$. 
\end{enumerate}
\end{thm}
\begin{rem}
Theorem \ref{BMelliptic} tells us that the Fourier--Mukai partner $S'$ of $S$ 
has an elliptic fibration $\pi'\colon S'\to C$.
Moreover, we can see $\kappa(S')=\kappa(S)\ne 0$ (cf.~\cite[Lemma 4.3]{BM01}). 
Then, as is explained in \S \ref{subsec:automorphism_eliiptic_surface},
$\pi'$ is a unique elliptic fibration structure on $S'$.
In particular, if two elliptic surfaces are mutually Fourier--Mukai partners, the base curves of elliptic fibrations
can be identified. We use this fact implicitly afterwards.  
\end{rem}

There are natural isomorphisms
$$
J_S(b)\cong J_S(b+\lambda_{S})\cong J_S(-b)
$$
over $C$ (see \cite[Remark 4.5]{BM01}).
Therefore, we can define the subset 
\begin{equation*}%\label{eqn:def_H}
H_S:=\{b\in (\Z/\lambda_{S}\Z)^* \mid J_S(b)\cong S \}
\end{equation*}
of the multiplicative group $(\Z/\lambda_{S}\Z)^*$.

\begin{cla}
For any pair of integers $b,c$, with $b,c\in\Z/\lambda_S\Z$, there is an isomorphism
$$
J_{J_S(c)}(b)\cong J_S(bc). 
$$ 
\end{cla}

\begin{proof}
Take an elliptic surface $B\to C$ with a section such that there is an isomorphism $\varphi_1\colon  J(S)\to B$ over $C$.
Let us set $\xi:=(S,\varphi_1)\in WC(B)$ (see \cite[\S2.2]{Ue11} for the definition of the Weil-Chatelet group $WC(B)$).
Then we see in \cite[\S2.2]{Ue11} that there is an isomorphism $\varphi_c\colon J(J_S(c))\to B$ over $C$ such that 
$(J_S(c),\varphi_c)$ corresponds to the element $c\xi\in WC(B)$.
By the same argument,  there are isomorphisms $\varphi'_{bc}\colon  J(J_{J_S(c)}(b))\to B$ and 
$\varphi_{bc}\colon  J(J_S(bc))\to B$ such that  both of $(J_{J_S(c)}(b),\varphi'_{bc})$ and $(J_S(bc),\varphi_{bc})$ 
correspond to the element $bc\xi$.
This finishes the proof.
\end{proof}

Since we have 
$$J_S(1)\cong S$$ 
(see, e.g.~\cite[\S2.2]{Ue11} and \cite[Remark 4.5]{BM01}),
we can see that the condition $J_S(b)\cong S$ implies that $J_S(c)\cong S$ for $c\in \Z$ with 
$bc\equiv 1$ ($\module m$). 
Therefore, it turns out that $H_S$ is a subgroup of $(\Z/\lambda_{S}\Z)^*$.
In particular, there is a natural 
one-to-one correspondence between 
the set $\FM (S)$ and the quotient group 
$(\Z/\lambda_{S}\Z)^* /H_S$. 

It is not easy  to describe the group $H_S$ concretely in general, 
which is equivalent to determine the set $\FM(S)$ (see  \cite{Ue04} and \cite{Ue11}). 
However when $\lambda_S\le 2$, $(\Z/\lambda_{S}\Z)^*$ is 
trivial, and hence $\FM (S)=\{S\}$. 

If $\lambda_S>2$, the group $H_S$ contains at least two elements 
$1,\lambda_S-1\in (\Z/\lambda_{S}\Z)^*$.
Hence, we have
\begin{equation*}%\label{eqn:FMeqn}
|\FM (S)|\le \varphi (\lambda_{S})/2,
\end{equation*}
where $\varphi$ is the Euler function. 
There are several examples in which 
we can compute the set $\FM (S)$ given in \cite[Example 2.6]{Ue11}.
In the upcoming paper \cite{Ue}, the author will also give examples
in which he can compute the set $\FM (S)$.
See \S \ref{sec:examples} for some more details.

%%%
%%%
%%%

\begin{rem}\label{rem:fibers}
Take a point $c\in C$ and an integer $b$ with $(b, \lambda_S)=1$. 
We know that there is an isomorphism $J(J_S(b))\cong J(S)$ 
over the curve $C$ (cf.~\cite[\S2.2]{Ue11}). 
Since it is known that the reduced form\footnote{If the fiber over the point $c$ is a multiple fiber of type 
${}_m\mathrm{I}_n$, then the reduced form is of type $\mathrm{I}_n$.} of the fibers 
of $S$ over the point $c$ is isomorphic to the fiber of $J(S)$ 
over $c$,
the same holds for the fibers of $S$ and $J_S(b)$.

Furthermore, 
the multiplicities of the fibers of 
$S$ and $J_S(b)$ over the same point are equal (see \cite[page 38]{Fr95} or the proof of \cite[Lemma 4.3]{BM01}).
Therefore, we conclude that the fibers on 
$S$ and $J_S(b)$ over the same 
point are isomorphic to each other.
\end{rem}

%%%%%%%%%%%%%%%%%%%%%%%%%%%%%%%%%%%%%%%%%%%%%%%%%%%%%%%%%%%%%%%%
%%%%%%%%%%%%%%%%%%%%%%%%%%%%%%%%%%%%%%%%%%%%%%%%%%%%%%%%%%%%%%%%
%%%%%%%%%%%%%%%%%%%%%%%%%%%%%%%%%%%%%%%%%%%%%%%%%%%%%%%%%%%%%%%%
\section{Autoequivalences of elliptic surfaces with non-zero Kodaira dimension}\label{sec:autoeq_non-zero_Kodaira}

\subsection{Notation and the setting}\label{subsec:general}
Let
$$
\pi\colon S\to C \quad \text{and}\quad \pi'\colon S'\to 
C
$$ 
be projective elliptic surfaces
with non-zero Kodaira dimension, and we denote the projections
by
$$
p\colon S\times S'\to S \quad \text{and}\quad p'\colon S\times S'\to S'.
$$
Let $\Phi=\Phi^{\mathcal{P}}\colon D(S)\to D(S')$ 
be a Fourier-Mukai transform. The following is well-known.

\begin{cla}\label{cla:FM_support}
For each $y\in S$, $\Supp \Phi(\mathcal{O}_y)$ is contained in a single fiber of $\pi'$.
Furthermore, for a point $x\in S$ with $\pi(x)\ne \pi(y)$, we have $\Supp\Phi(\mathcal{O}_y)\cap \Supp \Phi(\mathcal{O}_x)= \emptyset$.
\end{cla}

\begin{proof}
By the assumption on the Kodaira dimension, 
the cohomology class of $K_{S'}$ is a non-zero rational multiple of the cohomology class of a fiber of $\pi'$.
On the other hand, since the Serre functor commutes with the equivalence $\Phi$,
there is an isomorphism
$$
 \Phi(\mathcal{O}_y)\cong  \Phi(\mathcal{O}_y)\otimes \omega _{S'}.
$$
Furthermore, since $\Phi(\mathcal{O}_y)$ is simple,
we know that $\Supp \Phi(\mathcal{O}_y)$ is connected.
These facts imply that
$\Supp \Phi(\mathcal{O}_y)$ is contained in a single fiber of $\pi'$.

Suppose for a contradiction  that  $\Supp \Phi(\mathcal{O}_y)\cap \Supp \Phi(\mathcal{O}_x)\ne \emptyset$.
Take a point $z$ in this non-empty set.  
Then, it follows from Lemma \ref{lem:BM02} that  neither 
$\RHom (\Phi(\mathcal{O}_x),\mathcal{O}_z)$ nor $\RHom (\Phi(\mathcal{O}_y),\mathcal{O}_z)$
vanishes. Again Lemma \ref{lem:BM02} implies that $\Supp\Phi^{-1}(\mathcal{O}_z)$
contains the points $x$ and $y$.
This induces a contradiction, because $\Supp\Phi^{-1}(\mathcal{O}_z)$ is contained in a single fiber of $\pi$, and
$\pi(x)\ne \pi(y)$. 
Therefore, we obtain  $\Supp\Phi(\mathcal{O}_y)\cap \Supp \Phi(\mathcal{O}_x)= \emptyset$.
\end{proof}

Denote the inclusion $S'\cong y\times S'\hookrightarrow S\times S'$ by $i$. 
Notice that
\begin{equation}\label{eqn:LiP}
\mathcal{P}|_{y\times S'}\cong\Phi(\mathcal{O}_y)
\quad\text{and}\quad 
%\end{equation}
%
%and 
%
%\begin{equation}\label{eqn:SLiP}
\Supp \mathcal{P}|_{y\times S'}=i^{-1}(\Supp \mathcal{P})
\end{equation}
(see \cite[Lemma 3.29]{Hu06}).
Claim \ref{cla:FM_support} and \eqref{eqn:LiP} give 
$$
\dim \Supp \mathcal{P}=2  \text{ or }3.
$$ 

\begin{cla}\label{cla:isom_UVSC}
Let $V_1$ and $V'_1$ be non-empty open subsets of $C$.
Let us define $U_1:=\pi^{-1}(V_1)$ and  $U'_1:=\pi'^{-1}(V'_1)$. 
Suppose that there is an isomorphism $\varphi_{U_1}\colon U_1\to U'_1$.
Then there is an automorphism $\varphi_C\in \Aut C$ and an isomorphism $\varphi\colon S\to S'$ 
extending $\varphi_{U_1}$ such that $\varphi_C\circ \pi=\pi'\circ \varphi$ holds.  
\end{cla}

\begin{proof}
%We can see that the support of $\varphi_{U_1}^*(H')$ is contained in a finite union of fibers of $\pi$ for a base point free divisor $H'$ 
%defining the morphism $\pi'|_{U'_1}$.
%Hence it defines the morphism $\pi|_{U_1}$, and hence there is an isomorphism $\varphi_{V_1}\colon V_1\to V'_1$ such that 
%$\varphi_{V_1}\circ \pi|_{U_1}=\pi|_{U'_1}\circ \varphi_{U_1}$. 
We may assume that $V'_1\ne C$ (otherwise the proof is done).
The only proper connected subvarieties of $V'_1$ are points. 
Hence $\pi' |_{U'_1}\circ \varphi_{U_1}$ maps fibers of $\pi|_{U_1}$ to points and, consequently, 
$\varphi_{U_1}$ maps fibers to fibers. Thus there is a bijection
$\varphi_{V_1}\colon V_1\to V'_1$ such that 
$\varphi_{V_1}\circ \pi|_{U_1}=\pi'|_{U'_1}\circ \varphi_{U_1}$. 
One can deduce that $\varphi_{V_1}$ is a morphism from the fact that the composition 
$\varphi_{V_1}\circ \pi|_{U_1}$
%=\pi|_{U'_1}\circ \varphi_{U_1}$
is a morphism.

Since $C$ is a smooth projective curve, $\varphi_{V_1}$
extends to an isomorphism $\varphi_C$. 
Take the elimination of indeterminancies $S\gets  T\to S'$ of $\varphi _{U_1}$.
Then \cite[Proposition III.8.4]{BHPV} assures that the morphism $T\to C$ uniquely factors through a relatively minimal 
elliptic fibration. Namely, $\varphi _{U_1}$ extends an isomorphism $\varphi$ in the statement.
\end{proof}

Let us denote by $Z$  
the union of all $(-2)$-curves on $S$.
Note that the set 
$Z$ coincides with the union of all reducible fibers.
We also denote by $U$ the complement of $Z$ in $S$,
by $V$ the image of $U$ under $\pi$,
and by $F$ a smooth fiber of $\pi$. 
We define $Z',U',V'$
and $F'$ on $S'$ similarly.

Suppose that there is an isomorphism $\varphi \colon S\to S'$. Then, since $S$ and $S'$ have a unique elliptic fibration 
(see \S \ref{subsec:automorphism_eliiptic_surface}),
it induces  an automorphism $\varphi _C$ of $C$ which satisfies $\varphi_C\circ \pi=\pi'\circ \varphi$. 
Moreover, an isomorphism $\varphi_U\colon U\to U'$ extends to an isomorphism $\varphi \colon S\to S'$
by Claim \ref{cla:isom_UVSC}.
In particular, we have 
$$
\Aut S\cong \Aut U, \quad \Aut_CS\cong\Aut_VU.
$$

%%%%%%%%%%%%%%%%%%%%%%%%%%%%%%%%%%%%%%%%%%%%%%%%%%%%%%%%%%%%%%%%%
\subsection{Autoequivalences associated with reducible fibers}\label{subsec:reducible}
In this subsection, we show that every autoequivalence of $D(S)$ induces  an autoequivalence of $D_Z(S)$.

The following is crucial to show Proposition \ref{prop:Z}, the 
main result in \S\ref{subsec:reducible}.

%%%
%%%
%%%

\begin{lem}\label{lem:support}
%In the notation of \S\ref{subsec:general}, 
Take a point $x\in S$.  
\begin{enumerate}
\item
Suppose that there is an integer $i$ such that $H^i(\Phi (\mathcal{O}_x))\ne 0$ and  
$$c_1(H^i(\Phi(\mathcal{O}_x)))\cdot c_1(H^i(\Phi(\mathcal{O}_x)))=0.$$ 
Then $\Phi (\mathcal{O}_x)$ is a shift of a sheaf.
\item
If $x\in Z$, then  
$\Supp \Phi (\mathcal{O}_x)$ is contained in $Z'$.
\end{enumerate}
\end{lem}

\begin{proof}
(i)  By the equation \eqref{eqn:euler},
we have 
$$
\chi(H^i(\Phi(\mathcal{O}_x)),H^i(\Phi(\mathcal{O}_x)))=-c_1(H^i(\Phi(\mathcal{O}_x)))\cdot c_1(H^i(\Phi(\mathcal{O}_x)))=0,
$$
and hence
$$
\dim \Ext^1_{S'}(H^i(\Phi(\mathcal{O}_x)),H^i(\Phi(\mathcal{O}_x)))\ge 2.
$$
But then 
\cite[Lemma 2.9]{BM01} implies that $i$ is a unique 
integer such that $H^i(\Phi(\mathcal{O}_x))\ne 0$. Namely 
$\Phi(\mathcal{O}_x)$ is a shift of sheaf.

(ii) First  of all, we note that $\Phi(\mathcal{O}_y)$ is simple, because so is $\mathcal{O}_y$.
First let us consider the case $\dim \Supp \mathcal{P}=2$.
Then there is an irreducible component $W$ of $\Supp \mathcal{P}$
such that the restrictions $p|_{W}\colon W\to S$, $p'|_{W}\colon W\to S$ of projections $p,p'$ are birational morphism 
(see the proof of \cite[Theorem 2.3]{Ka02}). We put 
$$
q:=p'|_W\circ p|_W^{-1}\colon S\dashrightarrow S'.
$$ 
But as Kawamata pointed out in \cite[Lemma 4.2]{Ka02}, 
$q$ is isomorphic in codimension $1$, 
and hence in the surface case, it is an isomorphism by \cite[Ch.5, Lemma 5.1]{Ha77}.
Therefore, if a $(-2)$-curve $G$ contains the point $x$, 
the $(-2)$-curve $q(G)$ contains the point $q(x)\in\Supp \Phi (\mathcal{O}_x)$.
Because each $(-2)$-curve on $S'$ and 
$\Supp \Phi (\mathcal{O}_x)$
are always contained in a single fiber of $\pi'$,   
$\Supp \Phi (\mathcal{O}_x)$ is contained in the set $Z'$.
(Moreover we can see by \cite[Lemma 4.5]{Hu06} that $\Phi (\mathcal{O}_x)=\mc{O}_{q(x)}$.
We use this remark below.)

Next let us consider the case  $\dim \Supp \mathcal{P}=3$.
Suppose that $\Phi(\mathcal{O}_x)$ is a sheaf on $S'$, after replacing $\Phi$ with $\Phi\circ [n]$ for some $n\in \Z$.
Take a point $y$, which is sufficiently near the point $x$, but not in $Z$.  
Then $\Phi(\mathcal{O}_y)$ is also a sheaf by Lemma \ref{lem:FMfamily}, and
both of $\Phi(\mathcal{O}_x)$ and $\Phi(\mathcal{O}_y)$ are $1$-dimensional
by the assumption $\dim \Supp \mathcal{P}=3$.
Claim \ref{cla:FM_support} implies that we may assume 
that $\Phi(\mathcal{O}_y)$ is a sheaf on a smooth elliptic curve $F'$.
% Here, if necessary,we replace the point $y$ with another point.
It follows from \cite[Lemma 4.8]{IU05} that $\Phi(\mathcal{O}_y)$ is a $1$-dimensional simple $\mc{O}_{F'}$-module, and hence,
it is a locally free sheaf on a smooth elliptic curve $F'$. Then, it is known to be stable by \cite[Remark 3.4]{Br98}. 
Denote the 
Chern class of the stable sheaf $\Phi(\mathcal{O}_y)$ on the fiber $F'$ by $(0,aF',-b)$ for some integer $a,b$. Then we know that 
$(a\lambda_{S'},b)=1$ (see the proof of \cite[Proposition 4.4]{BM01}), and hence there is  
the elliptic surface $J_{S'}(a,b)\to C$ together with the universal sheaf $\mathcal{U}$ on $J_{S'}(a,b)\times S'$.
For the point $w\in J_{S'}(a,b)$ representing the stable sheaf 
$\Phi(\mathcal{O}_y)$, we have
$$
\mathcal{O}_w\cong (\Phi^{\mathcal{U}}_{J_{S'}(a,b)\to S'})^{-1}\circ\Phi(\mathcal{O}_y).
$$
It follows that the kernel of the Fourier--Mukai transform 
$(\Phi^{\mathcal{U}}_{J_{S'}(a,b)\to S'})^{-1}\circ\Phi$ has a
$2$-dimensional support.
Now, we can apply the case $\dim \Supp\mathcal{P}=2$ of the Lemma to see that 
there is a point $z$, contained in a $(-2)$-curve on $J_{S'}(a,b)$, such that 
$$
\mathcal{O}_z=(\Phi^{\mathcal{U}}_{J_{S'}(a,b)\to S'})^{-1}\circ\Phi(\mathcal{O}_x).
$$
% the object $\mc{E}:=(\Phi^{\mathcal{U}}_{J_{S'}(a,b)\to S'})^{-1}\circ\Phi(\mathcal{O}_x)$
% is supported on a reducible fiber of $J_{S'}(a,b)\to C$ over a point $c\in C$.
% By the definition of the universal family $\mc{U}$, the object 
% $\Phi^{\mathcal{U}}_{J_{S'}(a,b)\to S'}(\mc{E})\cong \Phi(\mathcal{O}_x)$ is supported on the fiber of $S'$ over the point $c$.
By Remark \ref{rem:fibers}, %this fiber is again reducible, Thus,
$\Supp\Phi^{\mathcal{U}}_{J_{S'}(a,b)\to S'}(\mathcal{O}_z)$  is contained in the set $Z'$.
Hence, so is $\Supp \Phi (\mathcal{O}_x)$.

Finally, suppose that  $\dim \Supp \mathcal{P}=3$ and
$\Phi(\mathcal{O}_x)$ is not a shift of a sheaf.
Take an integer $i$ such that $H^i(\Phi(\mathcal{O}_x))$ is non-zero.
Then the conclusion in (i) says that 
$c_1(H^i(\Phi(\mathcal{O}_x)))$ is not a multiple of a fiber $F'$. 
Because $H^i(\Phi(\mathcal{O}_x))$ is contained in a single fiber,
this means that $\Supp H^i(\Phi(\mathcal{O}_x))$ is contained in a reducible fiber, and hence in $Z'$. 
\end{proof}

Note that, if $\Supp\Phi (\mathcal{O}_x)$ is contained in an irreducible fiber of $\pi'$, the assumption in Lemma \ref{lem:support} (i) 
is satisfied by the equation \eqref{eqn:euler}.
 
Applying Lemmas \ref{lem:Z} and \ref{lem:support} (ii),
we obtain the following.

%%%
%%%
%%%

\begin{prop}\label{prop:Z}
%In the notation  of \S\ref{subsec:general},
There is a natural group homomorphism 
$$
\iota_Z:\Auteq D(S)\to \Auteq D_Z(S).
$$
\end{prop}

Let us define
$$
\Auteq ^\dagger D_Z(S):=\Image \iota_Z.
$$
 
The following is used in the proof of Lemma  \ref{lem:fiber_product}.  
\begin{cor}\label{cor:UU}
%In the notation  of  \S\ref{subsec:general},
Take a point $x\in U$. Then we have $$\Supp \Phi (\mc{O}_x)\subset U'.$$
\end{cor}

\begin{proof}
%Let us denote by $\Psi$ the quasi-inverse of $\Phi$.
We use Lemma \ref{lem:BM02}.
Take any point $y\in Z'$. Then we have
$$\RHom (\Phi (\mc{O}_x),\mc{O}_y)\cong \RHom (\mc{O}_x,\Phi^{-1}(\mc{O}_y))=0$$
 by Lemma \ref{lem:support} (ii).
This implies the assertion.
\end{proof}

%%%%%%%%%%%%%%%%%%%%%%%%%%%%%%%%%%%%%%%%%%%%%%%%%%%%%%%%%%%%%%%%%%%%%%%%

\subsection{Autoequivalences associated with  irreducible fibers: The classification}\label{subsec:non-reducible}

We begin \S\ref{subsec:non-reducible} by classifying Fourier--Mukai transforms between 
elliptic surfaces without reducible fibers.  

%%%
%%%
%%%

\begin{prop}\label{prop:noreducible}
Let 
$
\pi\colon S\to C$  and $\pi'\colon S'\to C$  
be elliptic surfaces without reducible fibers. Here we do not assume that $C$ is projective.
Let  $\Phi=\Phi^{\mathcal{P}}_{S\to S'}$ be 
a Fourier--Mukai transform %from $D(U)$ to $D(U')$
 over $C$ such that $\dim \Supp\mathcal{P}=2$ or $3$.
\begin{enumerate}
\item
The object $\mathcal{P}\in D(S\times_C S')$ 
is a shift of a sheaf, flat over $S$ by the first projection.
\item
The following are equivalent.
\begin{enumerate}
\item
%$\mathcal{P}$ is a sheaf on $S\times_C S'$ with
$\dim \Supp \mathcal{P}=2$.
\item
There are points $x\in S,y\in S'$
such that 
$\Phi (\mathcal{O}_x)\cong \mathcal{O}_y$.
\item There is a line bundle $\mathcal{L}$ on $S$ and 
an isomorphism $\varphi\colon S\to S'$ over $C$ such that 
$\Phi\cong \varphi_*((-) \otimes \mathcal{L})$.
\end{enumerate}
\item
Suppose that %$\mathcal{P}$ is a sheaf with 
$\dim \Supp \mathcal{P}=3$.
Then there are integers $a,b$ with $(a\lambda_{S'},b)=1$,
a universal sheaf $\mathcal{U}$ on $S'\times J_{S'}(a,b)$ and isomorphism 
$\phi\colon J_{S'}(a,b)\to S$ over $C$
such that 
$$
\Phi\cong \Phi^{\mathcal{U}}_{J_{S'}(a,b)\to S'}\circ \phi^*.
$$
\end{enumerate}
\end{prop}

\begin{proof}
(i)
Take any point $x\in S$.
By the irreducibility of fibers of $\pi'$ and Lemma \ref{lem:support} (i), 
the object $\Phi(\mathcal{O}_x)$ is a shift of a sheaf. 
Hence, \cite[Lemma 4.3]{Br99} %  and Lemma \ref{lem:FMfamily}
 implies that $\mathcal{P}$ is a shift of a sheaf, flat over $S$ by the first projection.

(ii)
As a consequence of (i), the dimension of the support of 
$\mathcal{P}|_{x\times S'}\cong\Phi(\mathcal{O}_x)$ 
does not depends on the choice of a point $x\in S$. 
Hence, in the situation (1) or (2),
every $\Phi (\mathcal{O}_x)$ has a finite support.
Then we get (3) by Lemma \ref{lem:birational}. Here, recall that 
$\mathcal{P}$ is a sheaf on $S\times_C S'$, and hence
$\varphi$ is defined over $C$.
Obviously, (3) implies (1) and (2).%, and (2) implies (1).

(iii)
The proof goes parallel to that of \cite[Proposition 4.4]{BM01}
(see also the proof of Lemma \ref{lem:support}).
Let us take a general point $x\in S$, and denote the 
Chern class of the stable sheaf $\Phi(\mathcal{O}_x)$ on a smooth fiber $F'$
by $(0,aF',-b)$ 
for some integers $a,b$. Then we know that 
$(a\lambda_{S'},b)=1$, and hence 
we can define
an elliptic surface $J_{S'}(a,b)\to C$ and a universal sheaf $\mathcal{U}$ on $S'\times J_{S'}(a,b)$.
For the point $y\in J_{S'}(a,b)$ representing a stable sheaf 
$\Phi(\mathcal{O}_x)$,
it is satisfied that 
$$
\Phi^{-1}\circ \Phi^{\mathcal{U}}_{J_{S'}(a,b)\to S'}(\mathcal{O}_y)\cong \mathcal{O}_x.
$$
Apply Lemma \ref{lem:birational}, and replace a universal 
bundle $\mathcal{U}$ with $\mathcal{U}\otimes p_{J_{S'}(a,b)}^*\mathcal{L}$
for a line bundle $\mathcal{L}$ on $J_
{S'}(a,b)$ if necessary,
then we obtain the assertion.
\end{proof}

If $\pi$ has a reducible fiber, the implication from (2) to (3) 
in Proposition \ref{prop:noreducible} (ii) fails because of the existence of twist functors associated to $(-2)$-curves.

Now we can prove the following important observation.

%%%
%%%
%%%

\begin{lem}\label{lem:fiber_product}
In the notation of \S\ref{subsec:general},  
there is an automorphism $\delta(\Phi)\in \Aut C$ satisfying 
$\delta(\Phi)(V)=V'$
such that $\mathcal{P}|_{U\times U'}$ is a shift of a coherent sheaf on $U_{\delta(\Phi)}\times_{V'}U'$.
Here $U_{\delta(\Phi)}\times_{V'}U'$ is the fiber product of
$U$ and $U'$ over $V'$ via the morphisms
$(\delta({\Phi})\circ \pi)|_U$ and $\pi'|_{U'}$.
\end{lem}

\begin{proof}
Below we freely compose $\Phi=\Phi^\mc{P}$ with a shift functor if necessary. 
First suppose that $\dim \Supp \mc{P}=2$.
Then for every $x\in U$, Corollary \ref{cor:UU}  implies that 
$\Supp\Phi(\mc{O}_x)$ is irreducible, and then, 
we obtain from Lemma \ref{lem:support} (ii) and  \cite[Lemma 4.3]{Br99} 
that $\mc{P}|_{U\times S'}$ is a sheaf, flat over $U$. 
In particular,  the sheaf $\Phi(\mc{O}_x)\cong\mc{P}|_{x\times U'}$ has finite support.
Then  \cite[Lemma 4.5]{Hu06} %Lemma \ref{lem:birational} 
%Then \cite[Lemma 4.5]{Hu06} and
%Corollary \ref{cor:UU} imply 
%that for any point $x\in U$,
implies that there is a point $y\in U'$ such that
$\Phi(\mc{O}_x)\cong \mathcal{O}_y.$
%\begin{equation*}\label{eqn:Ox&Oy}
%$\mc{P}|_{x\times U'}\cong \mathcal{O}_y.$
%\end{equation*}
%
% Let us regard $U'$ as $\Hilb^1(U')$, and $\mathcal{O}_\Delta$ as the universal sheaf on 
% $\Hilb^1(U')\times U'$, where $\Delta$ is the diagonal in $U'\times U'$.
% Then there is a morphism $\varphi_U \colon 
% U\to U'$ such that $(\varphi \times \id_{U'})^*\mathcal{O}_\Delta= \mathcal{P}|_{U\times U'}$. 
% Since $\Phi$ is an equivalence, we know that $x\in U$
% in \eqref{eqn:Ox&Oy} uniquely exists for a given $y\in U'$,
% and in particular  that  $\varphi_U$ is an isomorphism.  
Then \cite[Corollary 6.14]{Hu06} implies that there is a morphism $\varphi_U \colon  U\to U'$ 
satisfying
$
\Phi(\mc{O}_x)\cong \mc{O}_{\varphi_U(x)}.
$
Hence, we can apply Claim \ref{cla:isom_UVSC} %for $\varphi_U$ 
to obtain the automorphism $\varphi _C$ of $C$.
%Note that $(\varphi_U \times \id_{U'})^*\mathcal{O}_\Delta= \mathcal{P}|_{U\times U'}$,
% where $\Delta$ is the diagonal in $U'\times U'$.
Note that $\mathcal{P}|_{U\times U'}$ is the structure sheaf of the graph of the morphism  $\varphi_U$, and hence
a sheaf on $U_{\varphi_U}\times_{V'} U'$.
This $\varphi _C$ plays the role of $\delta({\Phi})$ in the assertion.

Next, we consider the case $\dim \Supp \mc{P}=3$.
By the same argument as in the proof of  
Proposition \ref{prop:noreducible}, 
there are integers $a,b$ with $(a\lambda_{S'},b)=1$
such that for every point $x\in U$, 
there is a point 
$y\in J_{U'}(a,b)$ satisfying
$$
\Phi^{-1}\circ \Phi^{\mathcal{U}}_{J_{S'}(a,b)\to S'}(\mathcal{O}_y)\cong \mathcal{O}_x.
$$
%Here   $\mathcal{U}$ is a universal sheaf on $J_{S'}(a,b)\times S'$ of a 
%fine moduli space $J_{S'}(a,b)$ for some $a,b\in\Z$ with $(a\lambda_{S'},b)=1$. 
%And 
Note that we can regard $J_{U'}(a,b)$ 
as the inverse image of $V'$ by the elliptic fibration 
$J_{S'}(a,b)\to C$. 
Let us denote by $\mc {Q}\in D(J_{S'}(a,b)\times S)$ the kernel of 
$\Phi^{-1}\circ \Phi^{\mathcal{U}}_{J_{S'}(a,b)\to S'}$.
Since $\dim \Supp \mc{Q}=2$, we can
apply the above argument to $\mc{Q}$,
and
then we obtain the assertion for 
$\mc{Q}|_{J_{U'}(a,b)\times U}$.
Since   $\mathcal{U}|_{J_{U'}(a,b)\times U'}$ is a coherent sheaf on $J_{U'}(a,b)\times_{V'}U'$, we obtain the assertion for 
$\mc{P}|_{U\times U'}$.
\end{proof}

\begin{rem}
For a point $x\in S$, we put $c:= \delta(\Phi)(\pi(x))\in C.$
Furthermore, if the point $x$ belongs to $U$,  we know from the definition of $\delta$ that
$\Supp \Phi(\mathcal{O}_x)$ is contained in the fiber ${\pi'}^{-1}(c)$.
Recall the facts 
that $\Supp \Phi(\mathcal{O}_x)=\Supp\mathcal{P}\cap (x\times S')$ by (\ref{eqn:LiP}), % and (\ref{eqn:SLiP}), 
 and  that $\Supp \Phi(\mathcal{O}_x)$ is contained in a single fiber.
These facts imply that $\Supp \Phi(\mathcal{O}_x)$ 
is contained in $\pi'^{-1}(c)$
for any $x\in S$.
Therefore, we conclude that $\mathcal{P}$ is an object of 
$D_{S_{\delta(\Phi)}\times_{C}S'}(S\times S')$. 

On the other hand,
$\mc{P}$ is not necessarily an object of $D(S_{\delta(\Phi)}\times_{C}S')$.
For instance, consider a twist functor $T_{\mc{O}_G}$ for a $(-2)$-curve $G$ on $S$.
Because the spherical object $T_{\mc{O}_G}(\mc{O}_G(2))$ is simple, the computation in Example \ref{exa:twist} (i) and 
\cite[Corollary 3.15]{Hu06} imply that it is not of a form $k_*\alpha$ 
for an object  $\alpha \in D(F_c)$, a fiber $F_c$ and the inclusion 
$k\colon F_c \hookrightarrow S$. Consequently, 
Lemma \ref{lem:base_change} (i) 
tells us that the kernel of $T_{\mc{O}_G}$ is not an object of 
$D(S\times_{C}S)$. Here note that $\delta (T_{\mc{O}_G})=\id _C$.
\end{rem}

%%%
%%%
%%%

\begin{prop}\label{prop:U}
In the notation of \S\ref{subsec:general}, assume furthermore
that $S=S'$. 
Then the Fourier--Mukai autoequivalence 
$\Phi=\Phi^\mc{P}$ of $D(S)$ 
induces a Fourier--Mukai autoequivalence 
$\iota_U(\Phi)$ of $D(U)$ over $V$, by restricting the kernel $\mc{P}$ to $U\times U$.  This $\iota_U$ defines a  
 group homomorphism 
$$
\iota_U\colon \Auteq D(S)\to \Auteq D(U)
$$
satisfying that
\begin{equation}\label{eqn:compatibility}
\mathbb{L} i^*\circ \Phi\cong \iota_U(\Phi)\circ \mathbb{L} i^*, 
\end{equation}
where  $i$ is 
the inclusion $U \hookrightarrow S$.
\end{prop}

\begin{proof}
The statement is a direct consequence of Lemma \ref{lem:fiber_product} and the isomorphisms 
$$
{\id}_U\cong  \iota_U(\Phi\circ \Phi^{-1}) 
\cong \iota_U(\Phi)\circ \iota_U(\Phi^{-1}).
$$ 
The isomorphism \eqref{eqn:compatibility} follows from 
a direct computation.
\end{proof}
Let us define 
$$
\Auteq^{\dagger} D(U):=\Image \iota_U.
$$
Note that the elements of $\Auteq^{\dagger} D(U)$ are classified in Proposition \ref{prop:noreducible}.

%%%
%%%
%%%

\begin{rem}\label{rem:image_delta}
Lemma \ref{lem:fiber_product} 
tells us that there is a group homomorphism 
\begin{equation*}\label{eqn:delta}
\delta\colon \Auteq D(S)\to \Aut C. 
\end{equation*}
The map $\delta$ factors through the map $\iota_U$, and 
hence it induces a map 
$$
\delta _U\colon \Auteq^\dagger D(U)\to \Aut V (\cong \Aut C).
$$
%
%We also denote it by $\delta$ 
%
%By definition, we have
%$\Image \delta_U=\Image\delta$.
%We also know that for projective elliptic surfaces
%$\pi\colon S\to C$ 
%and 
%$\pi'\colon S'\to C$ of non-zero Kodaira dimensions, 
%As is explained in \S\ref{subsec:general}, the isomorphism $\varphi \colon U\to U'$ extends the isomorphism $S\to S'$.
Then it follows from Proposition \ref{prop:noreducible} that
\begin{align*}
\Image\delta
=&
\left\{
\varphi_C \mid \varphi\in\Aut S, 
\text{ or } 
\varphi\colon S\to J_S(a,b) \text{ isomorphism with }
(a\lambda_S,b)=1
\right\}\\
=&
\left\{
\varphi_C \mid \varphi\in\Aut S, 
\text{ or } 
\varphi\colon S\to J_S(b) \text{ isomorphism with }
(\lambda_S,b)=1
\right\}.
\end{align*}
The second equality follows from \eqref{eqn:J(a,b)=J(b)}.
%It seems a difficult question to determine 
%the structure of $\Image\delta$ as well as 
%to determine the structure of the group 
%$H_S=\{ b\in (\Z/\lambda_S\Z)^* \mid J_S(b)\cong S \}$
%given in \S\ref{subsec:bridgeland}. 
\end{rem}

%%%%%%%%%%%%%%%%%%%%%%%%%%%%%%%%%%%%%%%%%%%%%%%%%%%%%%%%%%%%%%%%%%%%%%%%

\subsection{Autoequivalences associated with  irreducible fibers: The structure of the group}\label{subsec:non-reducible:structure}
We use the notation of \S\ref{subsec:general} in this subsection.
The aim of \S\ref{subsec:non-reducible:structure} is to study the structure of the group $\Auteq^\dagger D(U)$. 

For $m\in \Z$, we define the congruence subgroup 
of $\SL (2,\Z)$ by
$$
\Gamma_0(m):=\bigl\{ \begin{pmatrix}
c& a\\
d& b   
\end{pmatrix}\in \SL (2,\Z)\bigm|d\in m\Z \bigr\}.
$$
Let us consider the surjective map 
$$
\Gamma_0(\lambda_{S})\to (\Z/\lambda_S \Z)^*/H_S
\qquad  \begin{pmatrix}
c& a\\
d& b   
\end{pmatrix} \mapsto b.
$$
This is actually a group homomorphism and its kernel coincides 
with the group
$$
\bigl\{ \begin{pmatrix}
c& a\\
d& b   
\end{pmatrix}\in \Gamma_0(\lambda_{S}) \bigm| J_S(b)\cong S \bigr\}.
$$

%%%
%%%
%%%

\begin{thm}\label{thm:AeV_extension}
There is a short exact sequence
\begin{align*}
1\to 
\Span{\otimes \mathcal{O}_U(D)\mid D\cdot F=0}&\rtimes \Aut S\times \Z[2] 
\to
\Auteq ^\dagger D(U) \notag\\
&\to
\bigl\{ \begin{pmatrix}
c& a\\
d& b   
\end{pmatrix}\in \Gamma_0(\lambda_{S}) \bigm| J_S(b)\cong S \bigr\}
\to 1.
\end{align*}
\end{thm}

\begin{proof}
For $\Phi^\mathcal{P} \in \Auteq ^\dagger D(U)$, 
note that 
$\mathcal P$ is a shift of a sheaf on 
$U_{\delta_U(\Phi)}\times_{V}U$
by  Lemma \ref{lem:fiber_product}.

Take a smooth fiber $F=F_c$ over a point $c\in V$, and 
denote by $F'=F_{c'}$ a smooth fiber over the point $c'=\delta(\Phi)(c)$.
Let 
$k\colon F\hookrightarrow U, $ and 
$k'\colon F'\hookrightarrow U$ be the 
natural inclusions. 
 
It follows from Lemma \ref{lem:base_change} that 
we obtain a Fourier--Mukai transform  $\Phi^{\mathcal{P}|_{F\times F'}}_{F\to F'}$
such that there is an isomorphism of functors 
$$
\Phi^{\mathcal{P}|_{F\times F'}}\circ \mathbb{L}k^* \cong \mathbb{L}k'^*\circ  \Phi^\mathcal{P}.
$$
By choosing a basis as in \S\ref{subsec:elliptic},  there is an identification between $H^{\mathrm{ev}}(F,\Z)$ and  $H^{\mathrm{ev}}(F',\Z)$.
Then, we obtain a group homomorphism
%$$
%\epsilon\colon 
%\AeC^\dagger D(U) \to \Auteq D(F)\qquad \Phi^{\mathcal{P}} \longmapsto \Phi^{\mathcal{P}|_{F\times F}}.
%$$
%
%Let us consider the composition  
$$
\Theta_U\colon \Auteq^\dagger D(U) \to \SL (2,\Z).
$$
Notice that this morphism does not depend on the choice of a smooth fiber $F$ by the classification of the elements of $\Auteq^\dagger D(U)$ in Proposition \ref{prop:noreducible}.

Since $F$ and $F'$ are isomorphic, we fix an isomorphism.
Then $\Phi^{\mathcal{P}|_{F\times F'}}_{F\to F'}$ can be regarded as an autoequivalence of $D(F)$. Then we have
$$
\Theta_U (\Phi^{\mathcal{P}})=\theta(\Phi^{\mathcal{P}|_{F\times F'}}_{F\to F'}).
$$
This is important for the latter computation.
(Recall the definition of $\theta$ in \S\ref{subsec:elliptic}.)

Suppose that $\Phi=\Phi^{\mc{P}}\in \Ker \Theta_U$. 
Proposition \ref{prop:noreducible} forces that
$\dim \Supp \mathcal{P}=2$, and hence
$$
\Phi\in\Span{\otimes \mathcal{O}_U(D)\mid D\cdot F=0}\rtimes \Aut S\times  \Z[2] ,
$$
that is,
$$
\Ker \Theta_U\cong \Span{\otimes \mathcal{O}_U(D)\mid D\cdot F=0}\rtimes \Aut S\times  \Z[2].
$$

Next, let us consider the image of $\Theta_U$. 
Take integers $a,b$ with $a>0$ and $(a\lambda_S,b)=1$.
%that $J_S(b)\cong S$.
%the  image of $b$ in  $(\Z/\lambda _S \Z)^*$ belongs to the subgroup $H_S$ of $(\Z/\lambda _S \Z)^*$
%(recall the definition of $H_S$ in \S\ref{subsec:bridgeland}).
%In this case, we simply write $b\in H_S$. Then 
Assume that there is an isomorphism $\phi\colon J_S(a,b)\to S$.
Then
\begin{equation}\label{eqn:Phi^U}
\Phi^{\mathcal{U}}_{J_S(a,b)\to S}\circ \phi^*
\end{equation}
gives an autoequivalence of $D(S)$. In this case, it follows from 
Remark \ref{rem:d_divides_lambda} that
$$
\Theta_U\circ\iota_U
(\Phi^{\mathcal{U}}_{J_S(a,b)\to S}\circ \phi^*)
=\begin{pmatrix}
c& a\\
d& b   
\end{pmatrix}
\in \SL (2,\Z)
$$
holds 
for some $c,d\in \Z$ such that $\lambda_S$ divides $d$. 
We also have
$$
\Theta_U\circ\iota_U(\Pic S)
=\Span{
\begin{pmatrix}
1& 0\\
\lambda_{S}& 1   
\end{pmatrix}}
$$
and
$$
\Theta_U([1])=\begin{pmatrix}
-1& 0\\
0& -1   
\end{pmatrix}.
$$
Then it follows from Proposition \ref{prop:noreducible} that
\begin{align*}
&\Image \Theta_U\\
=&\Span{
\begin{pmatrix}
1& 0\\
\lambda_{S}& 1   
\end{pmatrix},
\begin{pmatrix}
c& a\\
d& b   
\end{pmatrix},
\begin{pmatrix}
-1& 0\\
0& -1   
\end{pmatrix}
\bigm| a\ne 0, bc-ad=1, d\in \lambda_{S}\Z, J_S(b)\cong S}.%\\
%=&\Span{
%\begin{pmatrix}
%1& 0\\
%\lambda_{S}& 1   
%\end{pmatrix},
%\begin{pmatrix}
%c& a\\
%d& b   
%\end{pmatrix}
%\bigm| a\ne 0, bc-ad=1, d\in \lambda_{S}\Z, J_S(b)\cong S}.
\end{align*}
Note that 
$$
\begin{pmatrix}
-1& 0\\
\lambda_{S}& -1   
\end{pmatrix}
=
\begin{pmatrix}
1& 0\\
-\lambda_{S}& 1   
\end{pmatrix}
\begin{pmatrix}
-1& 0\\
0& -1   
\end{pmatrix}
\in \Image \Theta_U.
$$
Furthermore, we can see 
$$
\begin{pmatrix}
-1& 0\\
0& -1   
\end{pmatrix}=
\begin{pmatrix}
-1& a\\
0& -1   
\end{pmatrix}
\begin{pmatrix}
1& a\\
0& 1   
\end{pmatrix}
$$
for any $a\in \Z$. 
Therefore, we conclude that 
$$
\Image \Theta_U
=\bigl\{ \begin{pmatrix}
c& a\\
d& b   
\end{pmatrix}\in \Gamma_0(\lambda_{S}) \bigm| J_S(b)\cong S \bigr\},
$$
which completes the proof of Theorem \ref{thm:AeV_extension}.
\end{proof}

We define 
$$
\Theta\colon \Auteq D(S)\to \SL (2,\Z)
$$ to be the composition $\Theta_U\circ \iota_U$, where
we regard $\iota_U$ as a surjective homomorphism from 
$\Auteq D(S)$ to $\Auteq^\dagger D(U)$.
Thus, by definition, $\Image \Theta=\Image \Theta_U$ holds.

\begin{rem}\label{rem:image_Theta}
\begin{enumerate}
\item
By the remark before Theorem \ref{thm:AeV_extension}, 
we have a group isomorphism
$$
\Gamma_0(\lambda_{S})/ \Image \Theta
\cong
(\Z/\lambda_S \Z)^*/H_S.
$$
As explained in \S\ref{subsec:bridgeland}, the latter group can be naturally identified with the set
$\FM(S)$.
\item
When $|(\Z/\lambda_S \Z)^*|\le 2$ (e.g. $\lambda_S \le 4$),
we have $(\Z/\lambda_S \Z)^*=H_S$. Hence,  
we can see that 
$
\Image \Theta=\Gamma_0(\lambda _S).
$
In particular, when $\lambda_S=1$, we see that 
$
\Image \Theta=\SL (2,\Z).
$
\end{enumerate}
\end{rem}

%%%%%%%%%%%%%%%%%%%%%%%%%%%%%%%%%%%%%%%%%%%%%%%%%%%%%%%%%%%%%%%%%%%
%\subsection{The image of $\epsilon$}
%
%%%%%%%%%%%%%%%%%%%%%%%%%%%%%%%%%%%%%%%%%%%%%%%%%%%%%%%%%%%%%%%%%%

\subsection{Kernels of $\iota_Z$ and $\iota_U$}
Now, let us study the kernel of the homomorphisms 
$\iota_Z$ 
given in Proposition \ref{prop:Z}.
First of all, we may assume that $Z\ne\emptyset$
because, otherwise, we have
$$\Ker\iota_Z= \Auteq D(S) =\Auteq^\dagger D(U),$$
and the last group was already  studied in 
\S\ref{subsec:non-reducible:structure}.
Take $\Phi\in \Ker \iota_Z$. Then for any $x\in Z$, 
we have 
%\begin{equation*}%\label{eq:pointZ}
$\Phi(\mathcal{O}_x)\cong\mathcal{O}_x.$
%\end{equation*} 
%
Hence, by \cite[Corollary 6.14]{Hu06}, there is a point 
$y\in U$ such that 
\begin{equation}\label{eq:pointU}
\Phi(\mathcal{O}_y)\cong\mathcal{O}_w
\end{equation} 
for some $w\in U$.
We can apply Proposition \ref{prop:noreducible} (ii)
for $\iota_U(\Phi)$ to obtain that, for all points $y\in U$, there is a point 
$w$ satisfying (\ref{eq:pointU}). 
Therefore, Lemma \ref{lem:birational} implies that $\Ker \iota_Z$
is contained in the group $A(S)$ of the standard autoequivalences (see \S \ref{subsec:motivation}), and hence
\begin{align}\label{ali:ker_iota_Z}
\Ker \iota_Z= \Span{\otimes \mathcal{O}_S(F_c) \mid F_c
\text{ a fiber of }\pi}
\rtimes \Aut_ZS,
\end{align} 
where %we define a normal subgroup of $\Aut S$ by
$
\Aut_ZS:=\{\varphi \in \Aut S \mid \varphi(z)=z\text{ for all }z\in Z \}.
$

Let us denote 
$$
B:=\Span{T_{\mathcal{O}_G(a)} \mid G \text{ is a $(-2)$-curve }}.$$
For the homomorphism $\iota_U$ given
in Proposition \ref{prop:U}, we believe that the equality
\begin{equation}\label{eq:B=ker_U}
B=\Ker\iota_U
\end{equation}
holds.
Actually we have the following.

%%%
%%%
%%%

\begin{lem}\label{lem:ker_iota_U}
Suppose that equality \eqref{eq:B=ker_U} holds.
Then Conjecture \ref{conj} is true.
\end{lem}

\begin{proof}
The result follows from the description of 
$\Auteq^\dagger D(U)$ in Theorem \ref{thm:AeV_extension}.
\end{proof}

In \S\ref{sec:I_n}, \ref{sec:-1} and \ref{sec:-2}, we shall check the equality \eqref{eq:B=ker_U} 
in the case that all reducible fibers on a given projective elliptic surface are of type $\mathrm{I}_n$ $(n>1)$, 
and consequently show  Conjecture \ref{conj} in this case.

%%%%%%%%%%%%%%%%%%%%%%%%%%%%%%%%%%%%%%%%%%%%%%%%%%%%%%%%%%
%%%%%%%%%%%%%%%%%%%%%%%%%%%%%%%%%%%%%%%%%%%%%%%%%%%%%%%%%%
%%%%%%%%%%%%%%%%%%%%%%%%%%%%%%%%%%%%%%%%%%%%%%%%%%%%%%%%%%

\section{Autoequivalences associated with singular fibers of type $\mathrm{I}_n$ for $n>1$}\label{sec:I_n}

Throughout this section, $\pi\colon S\to C$ is a projective 
elliptic surface whose reducible fibers are non-multiple cycles of $(-2)$-curves, that is, of type 
$\mathrm{I}_n$ for $n>1$. 
In this case, the set $Z$ is a disjoint union of  cycles of projective lines. Below
we regard $Z$ as a closed subscheme of $S$ equipped with  the reduced induced structure.

In our setting, line bundles on $(-2)$-curves are spherical
in the sense of Definition-Proposition \ref{def-prop:twist_functor}. 
Therefore, $\Auteq D_Z(S)$ contains twist functors, and hence it is highly involved.
The following is the main result of this article.

%%%
%%%
%%%

\begin{thm}\label{thm:typeI_n}
Let $S$ be a smooth projective elliptic surface with $\kappa (S)\ne 0$. 
Suppose that each reducible fiber  on the elliptic surface $S$ is  of type $\mathrm{I}_n$ for some $n>1$. 
Then Conjecture \ref{conj} is true. Namely
we have
\begin{align*}%\label{eqn:conj_Auteq_extension}
1\to \Span{B,\otimes \mathcal{O}_S(D)\mid D\cdot F=0}\rtimes &\Aut S\times  
\Z[2]
\to
\Auteq D(S) \notag\\
& \stackrel{\Theta}\to 
\bigl\{ \begin{pmatrix}
c& a\\
d& b   
\end{pmatrix}\in \Gamma_0(\lambda_{S}) \bigm| J_S(a,b)\cong S \bigr\}
\to 1.
\end{align*}
\end{thm}

Theorem \ref{thm:typeI_n} concerns the autoequivalences of the derived categories of surfaces containing  $\tilde{A}_n$-configurations of $(-2)$-curves. 
Many ideas of the proof come from \cite{IU05}, 
where we study the autoequivalences of the derived categories of surfaces
containing $A_n$-configurations of $(-2)$-curves.

%%%%%%%%%%%%%%%%%%%%%%%%%%%%%%%%%%%%%%%%%%%%%%%%%%%%%%%%%%%%%%%%%%%%%%%%%%%%%%%%%%%%%

\subsection{Reduction of the proof of Theorem \ref{thm:typeI_n}}
Recall that Lemma \ref{lem:ker_iota_U} tells us 
that, if we can show the equation \eqref{eq:B=ker_U},
we obtain Theorem \ref{thm:typeI_n}.

Let $\{ Z_j\}_{j=1}^M$ be the set of connected components of $Z$, that is, each $Z_j$ is a
singular fiber of type $\mathrm{I}_{n_j}$ for some 
$n_j>1$. We define 
$$
B_j:=\Span{T_{\mc{O}_G(a)}\mid G \text{ is a $(-2)$-curves contained in $Z_j$}}.
$$
Take a connected component of $Z$, and denote it by 
$
Z_0.
$ 
We put 
$$
Z_0=C_1\cup \cdots \cup C_n
$$ 
such that each $C_i$ is a $(-2)$-curve on $S$, and they satisfy
\begin{equation}\label{eqn:label}
C_l\cdot C_m = 
\begin{cases}
                1 & |l-m|=1, \text{ or }|l-m|=n-1 \\
                0 & \text{ otherwise}
\end{cases}
\end{equation}
in the case $n>2$. In the case $n=2$, $C_1$ and $C_2$ intersect each other transversely at two points.

The inclusion 
$
B\subset \Ker\iota_U
$
holds by Remark \ref{rem:twist}.
Thus in order to prove \eqref{eq:B=ker_U}, 
it is left to show 
$
 \Ker\iota_U\subset B.
$
This, in turn, can be reduced to showing the following.

%%%
%%%
%%%

\begin{prop}[cf.~Proposition 1.7 in \cite{IU05}]\label{proposition:step -2 of A_n}
%We use the notation of Theorem \ref{thm:typeI_n}. 
Suppose that we are given an autoequivalence $\Phi$ of $D_{Z_0}(S)$ preserving
the cohomology class $\ch(\mc{O}_x) \in H^4(S,\Q)$ for all points $x\in Z_0$.
Then, there are integers $a$, $b$ $(1\le b\le n)$ and $i$,
and there is an autoequivalence $\Psi\in B_0$ such that
\begin{equation}\label{eqn:standard}
\Psi \circ \Phi(\owe_{C_1}) \cong \owe_{C_{b}}(a)[i] 
\quad\mbox{and}\quad
\Psi \circ \Phi(\owe_{C_1}(-1)) \cong \owe_{C_{b}}(a-1)[i].
\end{equation}
Consequently, 
for any point $x\in C_1$, we can find a point $y\in C_b$ 
with $\Psi\circ\Phi (\mc{O}_x)\cong \mc{O}_y[i]$.  
\end{prop}
Suppose that we have shown Proposition \ref{proposition:step -2 of A_n}.
Note that  $\Ker \iota_U\cap \Ker\iota_Z =\{1\}$ by \eqref{ali:ker_iota_Z}.
Thus we can consider both, $\Ker \iota_U$ and $B$, as subgroups of 
$\Auteq D_Z(S)$ and prove the inclusion $\Ker\iota _U\subset B$ inside this group.
  
Take $\Phi \in \Ker \iota_U$. 
Since
$\delta(\Phi)=\id _C$ for $\delta$ as constructed in \S \ref{subsec:non-reducible}, 
%the induced autoequivalence $\iota_Z(\Phi)$ of $D_Z(S)$ 
$\Phi$ induces autoequivalences $\Phi_j$ of  $D_{Z_j}(S)$.  Since all points $x\in S$ define the same cohomology class $\ch(\mc{O}_x)$, 
the autoequivalence $\Phi_j$ satisfies the assumption in  Proposition \ref{proposition:step -2 of A_n}.
%Note that $\Phi$ satisfies the assumption that it preserves the cohomology class $\ch(\mc{O}_x)$
We fix some $j$ and put $n=n_j$ for simplicity. 
%Denote the induced autoequivalence by
%
%$$
%in \Auteq D_{Z_j}(S).
%$$
%
We take the irreducible decomposition of $Z_j$ as $Z_j=C_1\cup\cdots\cup C_{n}$.
Now we can 
apply Proposition \ref{proposition:step -2 of A_n} for $\Phi_j$ to find 
$
\Psi_j\in B_j
$
satisfing \eqref{eqn:standard}.
%as in Proposition \ref{proposition:step -2 of A_n}.
Since $\Psi_j \circ \Phi$ also belongs to $\Ker \iota_U$,
we have $b=1$ and $i=0$, and  $x=y$ in Proposition \ref{proposition:step -2 of A_n}. 
Hence, $\Psi_j \circ \Phi_j$ gives an autoequivalence of $D_{Z'_j}(S)$, where $Z'_j=C_2\cup\cdots\cup C_{n}$. 
Since $Z'_j$ is an
$A_{n-1}$-configuration  of $(-2)$-curves, \cite[Theorem 1.3]{IU05} implies that 
\begin{align*}
\Psi_j \circ \Phi_j\in &((\Span{B'_j, \Pic S} \rtimes \Aut S) \times \Z)\cap \Ker \iota_U \\
\cong & B'_j\rtimes \Span{\otimes \mathcal{O}_S(C_i)\mid i=1,\ldots, n}\\
\subset & B_j,
\end{align*}
where we put
$
B'_j:=\Span{T_{\mathcal{O}_{C_i}(a)}\mid i=2,\ldots,n}
$
and the last inclusion is a consequence of 
\cite[Proposition 4.18 (i)]{IU05}.
Hence, we know that 
$
\Phi_j\in B_j.
$

We apply this argument for each $Z_j$ ($1\le j\le M$), and 
then we can see that
$$
\Psi_1\circ \cdots \circ \Psi_M\circ \Phi\in B
$$
and hence
$\Phi \in B$. 
Therefore, we have
$
\Ker\iota_U\subset B
$
%The other inclusion follows from Remark \ref{rem:twist}.
%Because $\Ker \iota_U\cap \Ker\iota_Z =\{1\}$ by \eqref{ali:ker_iota_Z}, and  \eqref{eq:B=ker_U} 
as desired.

\begin{rem}
%For simplicity, assume that $Z$ is a connected.
%The group $\Auteq ^\dagger D_Z(S)$ acts on the Grothendieck group 
%$K_Z(S)$, which is freely generated by the classes  
%$[\mc{O}_{C_i}]$ and $[\mc{O}_x]$ for all $(-2)$-curves $C_i$ and some point $x\in Z$. Thus 
In the $A_n$-case of Proposition 4.2, that is \cite[Proposition 1.7]{IU05}, 
we do not need the assumption that  
$\Phi$ preserves the cohomology class $\ch(\mc{O}_x)$, since it is always true.
To the contrary, in the $\tilde{A}_n$-case, % in Proposition \ref{proposition:step -2 of A_n}
the image of $\Phi^{\mathcal{U}}_{J_S(a,b)\to S}\circ \phi^*$ 
in \eqref{eqn:Phi^U} under $\iota_Z$ does not preserve
the cohomology class $\ch(\mc{O}_x)$.
The existence of such elements 
forces us to put this assumption in Proposition \ref{proposition:step -2 of A_n}.
 
The subgroup
$$
\Span{B,{(\Pic S/\Span{\otimes\mathcal{O}_S(F_c) \mid  c\in V})}\rtimes (\Aut S/\Aut_ZS)\times \Z[1]}
$$
of $\Auteq ^\dagger D_Z(S)$ preserves $H^4(S,\Q)$, and therefore it is strictly smaller than $\Auteq ^\dagger D_Z(S)$.
This is in contrast to the $A_n$-case of Theorem \ref{thm:IU05}.
\end{rem}   
   
We shall prove Proposition \ref{proposition:step -2 of A_n} in \S\ref{sec:-2}.
As an intermediate step,
we first show  the following.

%%%
%%%
%%%

\begin{prop}[cf.~Proposition 1.6 in \cite{IU05}]\label{proposition:step -1 of A_n}
%We use the notation in Theorem \ref{thm:typeI_n}.
Let $\alpha$ be a spherical object in $D_{Z_0}(S)$.
Then there are integers $a$, $b$ $(1\le b\le n)$ and $i$,
and there is an autoequivalence $\Psi\in B_0$ such that
$$
\Psi(\alpha) \cong \owe_{C_b}(a)[i].
$$
\end{prop}

Proposition \ref{proposition:step -1 of A_n} is proved in \S\ref{sec:-1}.

%%%%%%%%%%%%%%%%%%%%%%%%%%%%%%%%%%%%%%%%%%%%%%%%%%%%%%%%%%

\subsection{The cohomology sheaves of spherical objects}

\paragraph{Torsion free sheaves on a chain of projective lines}
We consider the cycle of lines
$
Z_0=C_1\cup\cdots \cup C_n
$
as an abstract variety and 
denote it by $\tilde{\mb{I}}_n$.
The curves $C_i$'s are labelled as in \eqref{eqn:label} when $n>2$. 

We denote by $\mb{I}_n$ a chain of $n$ projective lines.
We put 
$\mb{I}_n=C'_1\cup \cdots \cup C'_n$ 
such that 
each $C'_i$ is a projective line, and they satisfy
$$
C'_l\cdot C'_m = 
\begin{cases}
                1 & |l-m|=1\\
                0 & \text{ otherwise}. 
\end{cases}
$$

For a coherent sheaf
$\mc{R}$ on $\mb{I}_n$ or $\tilde{\mb{I}}_n$, 
we denote by $\deg_{C}\mc{R}$
the degree of the restriction $\mc{R}|_{C}$ to the component  $C$ of $\mb{I}_n$ or $\tilde{\mb{I}}_n$.
It is known that a line bundle $\mc{L}$ on  $\mb{I}_n$ is determined by the degree $\mc{L}|_{C}$ on all
the components $C$, that is,
$$
\Pic \mb{I}_n\cong \Z ^n.
$$ 
The line bundle corresponding to the vector  
$(a_1, \dots, a_n)\in \Z^n$ is denoted by 
$$
\mc{O}_{\mb{I}_n}(a_1, \dots, a_n).
$$
When we write $*$ instead of $a_l$, we do not specify the degree at $C'_l$.
For instance, when we write 
$$
\mathcal{R}_1 = \owe_{\mb{I}_3}(a, b, *),
$$
this means that $\mc{R}_1$ is a line bundle on $
\mb{I}_3$
such that $\deg_{C'_1}\mc{R}_1=a$, $\deg_{C'_2}\mc{R}_1=b$ and $\deg_{C'_3}\mc{R}_1$ is arbitrary.
The expression
$$
\mc{R}_2=\mc{O}_{C'_1\cup \cdots}(a,*)
$$
means that  $\mc{R}_2= \owe_{\mb{I}_k}(a,*, \dots, *)$ for some (not further specified) $k \ge 2$.
Note that the support of $\mc{R}_2$ is strictly larger than $C'_1$.
We often use figures
\begin{align}\label{ali:circle}
\xymatrix@R=1ex@M=0ex{  & C'_1 & C'_2 & C'_3 \\
   \mc{R}_1:  & 
\setlength{\unitlength}{1ex}
\begin{picture}(2, 2)(-1, -1)
\put(0,0){\circle{2}}
\put(0,0){\makebox(0,0){\tiny$a$}}
\end{picture} 
\ar@{-}[r] &
\setlength{\unitlength}{1ex}
\begin{picture}(2, 2)(-1, -1)
\put(0,0){\circle{2}}
\put(0,0){\makebox(0,0){\tiny$b$}}
\end{picture}
\ar@{-}[r] &{\no} & \\
\mc{R}_2:  &
\setlength{\unitlength}{1ex}
\begin{picture}(2, 2)(-1, -1)
\put(0,0){\circle{2}}
\put(0,0){\makebox(0,0){\tiny$a$}}
\end{picture}
\ar@{-}[r] & & & }
\end{align}
\noindent
to define $\mc{R}_1, \mc{R}_2$ above.
We use a dotted line
\begin{align}\label{ali:circle2} 
\xymatrix@R=1ex@M=0ex{  &C'_1  & \\
          \mc{R}_3:    &\setlength{\unitlength}{1ex}
\begin{picture}(2, 2)(-1, -1)
\put(0,0){\circle{2}}
\put(0,0){\makebox(0,0){\tiny$a$}}
\end{picture}
\ar@{--}[r]& &}
\end{align}
to indicate that $\mc{R}_3$ is either $\mc{O}_{C'_1}(a)$ or $\mc{O}_{C'_1\cup\cdots}(a,*)$.

\paragraph{Torsion free, but not locally free sheaves on $\tilde{\mb{I}}_n$}
For any $m\in \Z$ satisfying $m- i \in n\Z$ with some $1\le i\le n$, we define  
$$C_m:=C_i.$$

%%%
%%%
%%%

\begin{prop}[Theorem 19 in \cite{BBDG}]\label{prop:BBDG}
If an indecomposable torsion free $\mc{O}_{\tilde{\mb{I}}_n}$-module 
$\mc{S}$ is not locally free, then there is a finite 
surjective morphism 
$$
p_k\colon \mathbb{I}_k \to \tilde{\mb{I}}_n,
$$ 
some integer $s$, and a line bundle $\mc{L}$ on $\mathbb{I}_k$ such that 
$
p_k(C'_l)=C_{l+s-1}
$
with 
$l=1,\ldots, k$, and
$$\mc{S}\cong p_{k*}\mc{L}.$$
\end{prop}
In the situation of Proposition \ref{prop:BBDG},
assume that 
$$
\mc{L}\cong\mc{O}_{\mb{I}_k}(a_s,\ldots , a_{s+k-1}).
$$
In this case, we set
$$
\mc{S}_s(a_s,\ldots , a_{s+k-1}):=\mc {S},\quad \text{ or }\quad
\mc{S}_{C_s\cup\cdots\cup C_{s+k-1}}(a_s,\ldots , a_{s+k-1}):=\mc {S}.
$$
We can see that
\begin{equation}\label{eqn:restriction_S_i}
\mc{S}_s(a_s,\ldots , a_{s+k-1})|_{C_m}\cong 
\bigoplus_{l\in m+n\Z,s\le l\le s+k-1}\mc{O}_{C_m}(a_{l}).
\end{equation}
Notice that, for $k<n$, we have
$$
\mc{S}_{C_1\cup\cdots \cup C_{k}}(0,\ldots,0)=\mc{S}_1(\overbrace{0,\ldots,0}^{k})\cong \mc{O}_{C_1\cup\cdots\cup C_{k}},
$$ 
but in contrast
$$
\mc{S}_{C_1\cup\cdots \cup C_n}(0,\ldots,0)=\mc{S}_1(\overbrace{0,\ldots,0}^{n})\not\cong \mc{O}_{\tilde{\mb{I}}_n}.
$$
%We also use  figures like \eqref{ali:circle}, \eqref{ali:circle2} for $\mc{S}_i(a_i,\ldots , a_{i+k-1})$.

\paragraph{The cohomology sheaves of spherical objects}
%Recall that $Z_0$ is a connected component of $Z$, 
%and it is a cycle of $n$ projective lines 
%labelled as in \eqref{eqn:label}, that is, $Z_0\cong\tilde{\mb{I}}_n$.

Henceforth, we freely use the notations and results on 
$\tilde{\mb{I}}_n$ mentioned above. 

For a complex analytic open subset $U$ of $S$ and a spherical object $\alpha\in D(S)$, 
let 
$$
\Sigma (\alpha)_{U}
$$
be the set of all indecomposable summands of the coherent $\mc{O}_U$-module $\bigoplus_p H^p(\alpha)|_{U}$. If $U=S$, we just denote it by 
$$\Sigma (\alpha).$$ 
We frequently use the following.

%%%
%%%
%%%

\begin{lem}\label{lem:cohom_spherical}
\begin{enumerate}
\item
For a spherical object $\alpha\in D_{Z_0}(S)$, the  direct sum 
$\bigoplus_p H^p(\alpha)$ of its cohomology sheaves is rigid as an $\mc{O}_S$-module, 
and a torsion free $\mathcal{O}_{Z_0}$-module. 
\item
Any $\mc{R}\in \Sigma(\alpha)$ cannot be a locally free 
$\mc O_{Z_0}$-module, and
it is of the form $\mc{S}_s(a_s,a_{s+1},\ldots , a_{t})$ for some integers 
$i,j$ with $C_{s-1}\ne C_t$.
\end{enumerate}
\end{lem}

\begin{proof}
(i) This is a direct consequence of  \cite[Proposition 4.5, Lemmas 4.8, 4.9]{IU05}.

(ii) 
For a torsion free sheaf $\mc{E}$ on $Z_0$ such that
 $c_1(\mc{E})$ is some multiple of $[Z_0]$,
we can see that $\chi _S(\mc{E},\mc{E})=-c_1(\mc{E})^2=0$, which implies that 
$\mc{E}$ is not rigid. Hence, the torsion freeness and the rigidity of $\mc{R}$ by (i) imply the result.
\end{proof}

\begin{rem}\label{rem:multiple}
Suppose that $S$ has a multiple fiber $mZ_0$, i.e. of type $_m\mathrm{I}_n$ for some $m>1, n>0$.
Then \cite[Lemma 4.8]{IU05} implies that $\bigoplus_p H^p(\alpha)$ is an $\mathcal{O}_{mZ_0}$-module.
In particular, we cannot conclude Lemma \ref{lem:cohom_spherical} (i).
This is the reason why we assume that 
each reducible fiber is non-multiple in Theorem \ref{thm:main}. 
\end{rem}

For a connected union of $(-2)$-curves 
$$
Z':=C_s\cup C_{s+1}\cup\cdots\cup C_t 
$$
contained in $Z_0$, take a sufficiently small complex analytic neighbourhood of $Z'$, and we denote it by 
$$
U_{s,\ldots,t}.
$$
Here, sufficiently small means that 
$$
Z_0\cap  U_{s,\ldots,t}\cong D_{s-1} \cup C_s\cup \cdots \cup C_t\cup D_{t+1}
$$
where $D_i$ is %an open complex submanifold of the projective line $C_i$ and isomorphic to 
a one-dimensional complex disk.

%For a torsion free sheaf $\mc{E}$ on $Z_0$ such that
 %$c_1(\mc{E})$ is some multiple of $[Z_0]$,
%we can see that $\chi _S(\mc{E},\mc{E})=-c_1(\mc{E})^2=0$, which implies that 
%$\mc{E}$ is not rigid. 
%In particular, any $\mc{R}\in \Sigma(\alpha)$ cannot be a locally free 
%$\mc O_{Z_0}$-module, and
%it is of the form $\mc{S}_s(a_s,a_{s+1},\ldots , a_{t})$ for some integers 
%$i,j$ with $C_{s-1}\ne C_t$.

For a given $\beta\in D_{Z_0}(S)$, define
$$
l_i(\beta):= \sum_p \length_{\mc{O}_{S, \eta_i}} H^p(\beta)_{\eta_i}
$$
for each curve $C_i$, 
where $\owe_{S,\eta _i}$ is the local ring of $S$ at the generic point $\eta _i$ of $C_i$,
$H^p (\beta)_{\eta_i}$ is the stalk over $\eta_i$
and $\length _{\owe_{S,\eta _i}}$ measures the length over $\owe_{S,\eta _i}$.
Furthermore we define 
$$
l(\beta):=\sum _{i=1}^nl_i(\beta),
$$
and also define 
$$
l(\beta|_U):=\sum _{i\colon U\cap C_i\ne \emptyset}l_i(\beta)
$$
for a complex analytic open subset $U$ of $S$.
These invariants play important roles in the induction step in the proofs
 of Propositions \ref{proposition:step -2 of A_n} and \ref{proposition:step -1 of A_n}.

Using Lemma \ref{lem:cohom_spherical}, we can deduce strong
conditions on the elements of $\Sigma(\alpha)$
 as Lemma \ref{lem:degree_restriction} below.
%A similar statement in the $A_n$-case to Lemma \ref{lem:degree_restriction} 
%is fully used in \cite{IU05}.
Before giving a general statement in Lemma \ref{lem:degree_restriction},
we first give a special one for the case $n=2$, since 
the proof is slightly different from the one in the case $n>2$.
%,which is slightly stronger than what we actually need afterwards.

%%%
%%%
%%%

\begin{lem}\label{lem:degree_restriction_n=2}
Let us consider the case $n=2$, i.e. $Z_0\cong \tilde{\mb{I}}_2$.
and let $\alpha$ be a spherical object in $D_{Z_0}(S)$.
Assume that there is an element
$$
\mc{S}:=\mc{S}_{s}(a_{s},a_{s+1},\ldots , a_{t})
\in \Sigma(\alpha)
$$ 
with $s\ne t$, which means that $l(\mc{S})>1$. 
Then there are integers 
$a$ and $i=1$ or $2$ satisfying the following conditions:
\begin{enumerate}
\item
We have $C_i=C_s=C_t$ and $a=a_s=a_t$.
%,\quad  \mbox{ and }\quad  a_j=a \mbox{ or } a+1$$ 
%for $j$ with $s<j<t$ and $C_i=C_j$. 
The integer $i$ (resp. The integer $a$) does not depend 
on the choice of $\mc{S}\in \Sigma (\alpha)$ (resp.  $\mc{S}\in \Sigma (\alpha)$ with $l(\mc{S})>1$).
\item
Assume that there is an element
$
\mc{S}'\in \Sigma(\alpha)
$ 
with $l(\mc{S}')=1$. Then 
$\mc{S}'=\mc{O}_{C_i}(b)$ with $b=a$ or $a-1$.
\item
In the situation of (ii), assume furthermore that the above $\mc{S}$ satisfies $l(\mc{S})>3$. Then 
we have 
\begin{equation*}\label{eqn:==}
a_{s+2}=a_{s+4}=\cdots=a_{t-2}=b+1.
\end{equation*}
%
%Moreover if we put $a':=a_{s+2}$, then $a'=a$ or $a+1$, and $b=a'-1$.
\end{enumerate}
\end{lem}

\begin{proof}
Take elements
$$
\mc{S}_1:=\mc{S}_{s_1}(a_{s_1},\ldots , a_{t_1}),\quad \mc{S}_{2}:=\mc{S}_{s_2}(b_{s_2},\ldots , b_{t_2})
\in \Sigma(\alpha).
$$
First, we show  $C_{s_1}=C_{t_2}$. To the contrary, suppose that $C_{s_1}\ne C_{t_2}$, namely  $C_{s_1}=C_{t_2+1}$ holds. 
Then there is a non-split exact sequence
$$0\to \mc{S}_2\to \mc{S}_{s_2}(b_{s_2},\ldots ,b_{t_2-1}, b_{t_2}+1,a_{s_1},\ldots , a_{t_1})\to \mc{S}_1\to 0,$$
which gives a contradiction with the rigidity of 
$\mc{S}_1\oplus \mc{S}_2$ in Lemma \ref{lem:cohom_spherical} (i). 
Thus, we obtain  $C_{s_1}=C_{t_2}$.
 
 If we replace the role of $\mc{S}_1$ and $\mc{S}_2$ in the above argument, then  
 we obtain  $C_{s_2}=C_{t_1}$. Furthermore, consider the special case $\mc{S}_1=\mc{S}_2$. 
 Then  this particularly implies the equalities 
\begin{equation}\label{eqn:C_s=C_t}
C_{s_1}=C_{t_1}=C_{s_2}=C_{t_2}.
\end{equation}
Hence, we obtain the assertion $C_i=C_s=C_t$ in (i), and 
the assertion $\Supp \mc{S}'=C_i$ in (iii).
The equalities \eqref{eqn:C_s=C_t} and the rigidity of $\mc{S}_1\oplus \mc{S}_2$
also imply that 
\begin{equation}\label{eqn:S_1_S_2}
2=-c_1(\mc{S}_1)c_1(\mc{S}_2)=\chi_S(\mc{S}_1,\mc{S}_2)= 
\dim\Hom_S (\mc{S}_1,\mc{S}_2)+\dim\Ext^2_S(\mc{S}_1,\mc{S}_2).
\end{equation}
Note that $\Ext^2_S(\mc{S}_1,\mc{S}_2)\cong \Hom_S (\mc{S}_2,\mc{S}_1)^\vee$ by the Grothendieck--Serre duality.

(i)
Now suppose that $l(\mc{S}_1)>1$ and $l(\mc{S}_2)>1$. 
If $a_{s_1}> a_{t_1}$,
there is a morphism
$$
\mc{S}_1\twoheadrightarrow \mc{O}_{C_{t_1}}(a_{t_1})= \mc{O}_{C_{s_1}}(a_{t_1})\hookrightarrow\mc{S}_{1}.
$$
And hence, we have $\dim \Hom_S(\mc{S}_1,\mc{S}_1)\ge 2$, which contradicts \eqref{eqn:S_1_S_2} in the case $\mc{S}_1=\mc{S}_2$.
Then, we conclude that $a_{s_1}=a_{t_1}$ and
$b_{s_2}=b_{t_2}$. 

Furthermore, if $a_{s_1}\ne b_{s_2}$, we see that
$$
\dim\Hom_S (\mc{S}_1,\mc{S}_2)\ge 4 \mbox{ or } 
\dim\Ext^2_S(\mc{S}_1,\mc{S}_2)\ge 4,
$$ 
which gives a contradiction
with \eqref{eqn:S_1_S_2}.

(ii) Assume that $b>a$.  Then $\dim\Hom_S(\mc{S},\mc{S}')\ge 4$. Similarly, if $b<a-1$,
we have  $\dim\Hom_S(\mc{S}',\mc{S})\ge 4$. Hence, both possibilities contradict \eqref{eqn:S_1_S_2}, and we conclude that $b=a$ or $a-1$ as asserted.

%The rigidity of $\mc{S}\oplus\mc{S}'$ implies that $\mc{S}'=\mc{O}_{C_i}(a+1)$, $\mc{O}_{C_i}(a)$ or $\mc{O}_{C_i}(a-1)$,
%and that $|a_j-a|\le 1$ for all $j-i\in 2\Z$
%(see Lemma \ref{lem:degree_restriction} (iii-1)).
%Then the equality \eqref{eqn:S_1_S_2} implies the assertions (ii) and (iii).

(iii) Suppose that $b=a$ in (ii). Then  by  \eqref{eqn:S_1_S_2}  we see
$$
\dim \Hom_S(\mc{S},\mc{S}')=2 \text{ and } \dim \Hom_S(\mc{S}',\mc{S})=0.
$$ This impies the conclusion.
Next, suppose that $b=a-1$. Then, 
 \eqref{eqn:S_1_S_2} implies 
$$
\dim \Hom_S(\mc{S},\mc{S}')=0 \text{ and } \dim \Hom_S(\mc{S}',\mc{S})=2.
$$
This implies the conclusion.
\end{proof}

Let us proceed a general statement for any $n$.

%%%
%%%
%%%

\begin{lem}\label{lem:degree_restriction}
Let $\alpha$ be a spherical object in $D_{Z_0}(S)$.
Take elements
$$
\mc{S}_1:=\mc{S}_{s_1}(a_{s_1},\ldots , a_{t_1}),\quad \mc{S}_{2}:=\mc{S}_{s_2}(b_{s_2},\ldots , b_{t_2})
\in \Sigma(\alpha).
$$ 
\begin{enumerate}
\item
Let us take a reduced closed subscheme $Z'=C_i\cup C_{i+1}\cup\cdots \cup C_j$ of $Z_0$ for some $i,j\in\Z$ with 
$0\le j-i\le n-1$.
Then
$(\mc{S}_1\oplus 
\mc{S}_2)|_{Z'}$ is a rigid  $\mc{O}_S$-module. 
\item
We have $C_{t_1}\ne C_{s_2-1}$.
\item
For integers
$l,m$ satisfying
 $s_1\le l\le t_1$ and $s_2\le m\le t_2$ such that 
$C_l=C_m$, we have the following.
\begin{enumerate}
\item
$
|a_l-b_m|\le 1.
$
\item
Suppose that $s_1< l\le t_1$ and $s_2= m\le t_2$:
\[ 
\xymatrix@R=1ex@M=0ex{  & C_{l-1} & C_l & C_{l+1} \\
\text{A part of }   \mc{S}_1:  & 
\setlength{\unitlength}{1ex}
{}
\ar@{-}[r] &
\setlength{\unitlength}{1ex}
\begin{picture}(2, 2)(-1, -1)
\put(0,0){\circle{2}}
\put(0,0){\makebox(0,0){\tiny$a_l$}}
\end{picture}
\ar@{--}[r] & & \\
\text{The beginning of } \mc{S}_2:  &
&
\setlength{\unitlength}{1ex}
\begin{picture}(2, 2)(-1, -1)
\put(0,0){\circle{2}}
\put(0,0){\makebox(0,0){\tiny$b_m$}}
\end{picture}
\ar@{--}[r] & &  }
\]

Then we have
$a_l\ge b_m$.
\item
Suppose that $s_1< l< t_1$ and $s_2= m= t_2$: 
\[ 
\xymatrix@R=1ex@M=0ex{  & C_{l-1} & C_l & C_{l+1} \\
\text{A part of }   \mc{S}_1:  & 
\setlength{\unitlength}{1ex}
{}
\ar@{-}[r] &
\setlength{\unitlength}{1ex}
\begin{picture}(2, 2)(-1, -1)
\put(0,0){\circle{2}}
\put(0,0){\makebox(0,0){\tiny$a_l$}}
\end{picture}
\ar@{-}[r] & & \\
 \mc{S}_2:  &
&
\setlength{\unitlength}{1ex}
\begin{picture}(2, 2)(-1, -1)
\put(0,0){\circle{2}}
\put(0,0){\makebox(0,0){\tiny$b_m$}}
\end{picture}
& &  }
\]
Then we have
$a_l= b_m+1$.
\item
Suppose that $s_1< l= t_1$ and $s_2= m< t_2$.
\[ 
\xymatrix@R=1ex@M=0ex{  & C_{l-1} & C_l & C_{l+1} \\
\text{The end of }   \mc{S}_1:  & 
\setlength{\unitlength}{1ex}
{}
\ar@{-}[r] &
\setlength{\unitlength}{1ex}
\begin{picture}(2, 2)(-1, -1)
\put(0,0){\circle{2}}
\put(0,0){\makebox(0,0){\tiny$a_l$}}
\end{picture}
 & & \\
\text{The beginning of }  \mc{S}_2:  &
&
\setlength{\unitlength}{1ex}
\begin{picture}(2, 2)(-1, -1)
\put(0,0){\circle{2}}
\put(0,0){\makebox(0,0){\tiny$b_m$}}
\end{picture}\ar@{-}[r]
& &  }
\]
Then we have
$a_l= b_m$.
\end{enumerate}
\end{enumerate}
\end{lem}

\begin{proof}
(i)
We may assume that $Z'\ne Z$.
Let us consider the restriction map
\begin{equation*}
{S}_1\oplus\mc{S}_2\to (\mc{S}_1\oplus 
\mc{S}_2)|_{Z'}.
\end{equation*}
Note that there are no homomorphism from its kernel to $(\mc{S}_1\oplus \mc{S}_2)|_{Z'}$, since their supports intersect at finitely many points.
Then we can apply Mukai's lemma (see \cite[Lemma 2.2 (2)]{KO95}), and hence
the rigidity of $\mc{S}_1\oplus 
\mc{S}_2$ implies the rigidity of  $(\mc{S}_1\oplus 
\mc{S}_2)|_{Z'}$ as an $\mc{O}_S$-module.

(ii)  To the contrary, assume that $C_{t_1}=C_{s_2-1}$.
Then we can actually deduce a contradiction in a completely similar way to 
the one in the proof of $C_s=C_t$ in Lemma \ref{lem:degree_restriction_n=2} (i). 
%But for $n>2$, we can also deduce a contradiction as follows:
%Apply (i) for $Z'=C_{t_1}\cup C_{s_2}$, and 
%then the rigidity of $(\mc{S}_1\oplus \mc{S}_2)|_{Z'}$ follows.
%This contradicts with the fact that 
%$\mc{S}_1|_{Z'}$ and $\mc{S}_2|_{Z'}$ contain $\mc{O}_{C_{t_1}}(a_{t_1})$
% and $\mc{O}_{C_{s_2}}(a_{s_2})$ 
% respectively as direct summands.

(iii)
Take $Z'=C_l$ in (i) for (1), and $Z'=C_{l-1}\cup C_l$ in (i) for (2) in the case $n>2$.   
Then (1) follows immediately from the equation \eqref{ali:ext^1} in \S \ref{subsec:euler_form} and the rigidity of 
$$
\mc{O}_{C_l}(a_l)\oplus \mc{O}_{C_m}(b_m)=\mc{O}_{C_l}(a_l)\oplus \mc{O}_{C_l}(b_m),
$$ 
which is a direct summand of the rigid sheaf $(\mc{S}_1\oplus \mc{S}_2)|_{C_l}$.
In the situation of (2) in the case $n>2$, we know the rigidity of 
$$
\mc{O}_{C_{l-1}\cup C_{l}}(a_{l-1},a_{l})\oplus \mc{O}_{C_{m}}(b_m)=\mc{O}_{C_{l-1}\cup C_{l}}(a_{l-1},a_{l})\oplus \mc{O}_{C_{l}}(b_m).
$$ 
Then, there is a surjection 
$$\Hom _S(\mc{O}_{C_l}(b_m),\mc{O}_{C_l}(a_l))\to\Ext^1 _S(\mc{O}_{C_l}(b_m),\mc{O}_{C_{l-1}}(a_{l-1}-1)).$$
The non-vanishing of the latter space forces the result.
The statement of (2) in the case $n=2$ is proved in Lemma \ref{lem:degree_restriction_n=2}.

We leave it to the reader to show (3) since all the ideas have already appeared. 
%Taking $Z'=C_{l-1}\cup C_l\cup C_{l+1}$ in (i) for (3) in the case $n>3$, then 
%we can prove (3) in the case $n>3$ similarly. The statement (3) in the case $n=2$ is shown in Lemma \ref{lem:degree_restriction_n=2} (iii).
%We leave the reader to show (3) in the case $n=3$, since all ideas have already appeared in the other cases. 
The statement (4) is a direct consequence of (2).
\end{proof}

%%%%%%%%%%%%%%%%%%%%%%%%%%%%%%%%%%%%%%%%%%%%%%%%%%%%%%%%%%
%%%%%%%%%%%%%%%%%%%%%%%%%%%%%%%%%%%%%%%%%%%%%%%%%%%%%%%%%%
%%%%%%%%%%%%%%%%%%%%%%%%%%%%%%%%%%%%%%%%%%%%%%%%%%%%%%%%%%  

\section{The proof of Proposition \ref{proposition:step -1 of A_n}}\label{sec:-1}

Suppose that we are given a spherical object 
$\alpha \in D_{Z_0}(S)$ with $l(\alpha)=1$. Then we get $\alpha\cong\owe _{C_b}(a)[i]$ for some $a,b,i\in\Z$.
(Or, assume that $l(\alpha)<n$. Then $\Supp \alpha\ne Z_0$, and then 
\cite[Proposition 1.6]{IU05} implies Proposition \ref{proposition:step -1 of A_n}.)
If we prove that for a spherical $\alpha$ with $l(\alpha)>1$, there is
an autoequivalence $\Psi\in B_0$ such that 
\begin{equation}\label{eqn:l(alpha)}
l(\Psi (\alpha)) < l(\alpha),
\end{equation} 
then, since $\Psi(\alpha)$ is again spherical, the induction on $l(\alpha)$ yields Proposition \ref{proposition:step -1 of A_n}.
On the other hand, we can show the inequality
\begin{equation*}%\label{eqn:l(H(alpha))}
l(\Psi(\alpha))\le\sum_q l(\Psi (H ^q(\alpha)))
\end{equation*}
by a similar way to 
\cite[Lemma 4.11]{IU05} for any $\Psi \in \Auteq D_{Z_0}(S)$. Thus to
get \eqref{eqn:l(alpha)}, 
it is enough to show that
\begin{equation}\label{eqn:l_cohom}
\sum_q l( \Psi (H^q(\alpha)))< \sum_q 
l(H ^q(\alpha))(=l(\alpha)).
\end{equation}
%
%We prove Proposition \ref{proposition:step -1 of A_n} in this section.
%Throughout this section, we assume that $Z$ is connected as in Proposition \ref{proposition:step -1 of A_n}. 
%%%%%%%%%%%%%%%%%%%%%%%%%%%%%%%%%%%%%%%%%%%%%%%%%%%%%%%%%%

\subsection{Auxiliary results}

Let us begin the following lemma. 
%%%
%%%
%%%

\begin{lem}\label{lem:restriction3}
Take integers $s,t$ with $s\le t$, a complex analytic open subset
$U=U_{s,\ldots,t}$ and  
$\Psi \in \Span{T_{\owe_{C_i}(a)}\mid s\le i\le t, a\in \Z}$. 
Let $\beta$ be an object of $D_{Z_0}(S)$.
Then we have the following:
\begin{enumerate}
\item
$\Psi(\beta)|_{U}\cong \Psi(\beta|_{U})$, and
\item
$
l(\Psi(\beta))-l(\beta)
=
l(\Psi(\beta)|_{U})-l(\beta|_U).
$
\end{enumerate}
\end{lem}

\begin{proof}
It is enough to consider the case $\Psi=T_{\owe_{C}(a)}$ for some $a$ and some $C=C_i$ with $s\le i\le t$.

(i)
%Let us consider  the exact triangle 
%$$
%\RGamma _{S\backslash U}(S,{\mc{O}_{C}(a)}^\vee\Lotimes \beta)\to \RGamma (S,{\mc{O}_{C}(a)}^\vee\Lotimes \beta)\to
%\RGamma (U,{\mc{O}_{C}(a)}^\vee\Lotimes \beta),
%$$
%where 
%$\RGamma _{S\backslash U}$ is the derived functor of 
%the left exact functor $\Gamma_{S\backslash U}$ taking the global sections whose support lies in $S\backslash U$. 
%The first term vanishes, since 
%$\Supp {\mc{O}_{C}(a)}^\vee\Lotimes \beta\subset C$.
%Hence the second morphism is isomorphic.
%Consider the morphism between exact triangles in $D(U)$:
Since $\Supp {\mc{O}_{C}(a)}= C\subset U$, we have 
$$
\RHom _{D(S)}(\mc{O}_{C}(a),\beta)\cong \RHom _{D(U)}(\mc{O}_{C}(a),\beta|_U).
$$ 
Hence, there is the isomorphism
of exact triangles in D(U):
\[
\xymatrix{
\RHom _{D(S)}(\mc{O}_{C}(a),\beta)\otimes_{\mb{C}} \mc{O}_{C}(a) \ar[r]\ar[d]^{\cong} 
& \beta|_U \ar[r]\ar@{=}[d] 
& T_{\mc{O}_{C}(a)}(\beta)|_U\ar[d]^{\cong} \\
\RHom _{D(U)}(\mc{O}_{C}(a),\beta|_U)\otimes_{\mb{C}} \mc{O}_{C}(a)\ar[r]
& \beta|_U \ar[r]
& T_{\mc{O}_{C}(a)}(\beta|_U)
}
\]
%The left vertical arrow is isomorphic as above, and thus 
%so is the right one.

(ii) From the exact triangle \eqref{eqn:twist_triangle}, 
 it is easy to see the equality
$$
(T_{\mc{O}_{C}(a)}(\beta))|_{S\backslash C}=
\beta|_{S\backslash C}.
$$
Hence we have
$$
\sum _{i\colon U\cap C_i= \emptyset}l_i(T_{\mc{O}_{C}(a)}(\beta))=
\sum _{i\colon U\cap C_i= \emptyset}l_i(\beta).
$$
Since by definition of $l(\beta|_U)$ we have 
$$
l(\beta)=l(\beta|_U)+
\sum _{i\colon U\cap C_i= \emptyset}l_i(\beta),
$$
we obtain 
$$
l(T_{\owe_{C}(a)}(\beta))-l(\beta)
=
l((T_{\mc{O}_{C}(a)}(\beta))|_{U})-l(\beta|_U).
$$
as desired
%From (i), the right-hand side equals to
%$l(T_{\mc{O}_{C}(a)}(\beta|_{U}))-l(\beta|_U)$,
%which gives the assertion.
\end{proof}

By Lemma \ref{lem:restriction3} (i), 
we can use the computation in \cite[Lemma 4.15]{IU05} in our setting.
For example, when $n=3$, taking $U=U_{0,1}$ and $U_{1,2}$ in Lemma \ref{lem:restriction3} (i), we have
$
T_{\mc{O}_{C_1}(-1)}(\mc{S}_{C_1\cup C_2\cup C_3\cup C_1}(0,0,0,0))=\mc{O}_{C_2\cup C_3}.
$

Lemma \ref{lem:restriction3} (ii) is useful to prove \eqref{eqn:l_cohom} as it allows to apply many results from the 
$A_n$-case of  \cite{IU05} to our $\tilde{A}_n$-case. Namely, let $\mc{S}\in\Sigma(\alpha)$ and $U=U_{s,\ldots,t}$ as above.
Then there is a smooth surfaces $\hat{S}$ containing an $A_{t-s+3}$-configuration $\hat{Z}_0$ and an open subset $\hat{U}\cong U$ 
such that $\hat{U}\cap\hat{Z}_0\cong U\cap Z$ together with a sheaf $\hat{\mc{S}}$ on $\hat{S}$ such that 
$\hat{\mc{S}}|_{\hat{U}}\cong \mc{S}|_U$. Now applying 
Lemma \ref{lem:restriction3} (ii) twice, we get the equality
$$
l(\Psi(\mc{S}))-l(\mc{S})=l(\Psi(\hat{\mc{S}}))-l(\hat{\mc{S}}).
$$
The right hand side of this equation is computed in many cases in \cite{IU05}, see in particular \cite[Lemma 4.15]{IU05}. In the following,
we will often refer the reader to statements in \cite{IU05} and claim that the proof in our case is analogous. Note that in many steps of 
the proof in \cite{IU05}, 
we referred to \cite[(6.2)]{IU05} and left the details to the reader. 
Many of the details (or rather their analogues in the $\tilde{A}_n$-case) are stated explicitly 
in Lemmas  \ref{lem:degree_restriction_n=2} and \ref{lem:degree_restriction} of the present paper.

To prove Lemma A below, we need local versions of \cite[Lemmas 6.3, 6.6]{IU05};

%%%%
%%%%
%%%% lemma:fundamental

\begin{lem}[cf.~Lemma 6.3 in \cite{IU05}]\label{lemma:fundamental}
Let $\alpha \in D_{Z_0}(S)$ be a spherical object and $C=C_s \subset Z_0$ a $(-2)$-curve.
Take a sufficiently small complex analytic 
neighbourhood $U=U_s$ 
of $C$. 
Assume that for every $p$ we have a decomposition
$$
H^p(\alpha)|_U=\bigoplus_j^{r_1^p}\mathcal{R}_{1,j}^p
\oplus\bigoplus_j^{r_2^p}\mathcal{R}_{2,j}^p
\oplus\bigoplus_j^{r_3^p}\mathcal{R}_{3,j}^p
\oplus\bigoplus_j^{r_4^p}\mathcal{R}_{4,j}^p,
$$
where $\mathcal{R}^p_{k,j}$'s are sheaves of the forms:
\[ 
\xymatrix@R=1ex@M=.3ex{
                           & & C &  \\ 
          \mathcal{R}_{1,j}^p:& \ar@{-}[r]&{\zero}\ar@{-}[r]&\\
          \mathcal{R}_{2,j}^p:& &{\zero}\ar@{-}[r] &\\
          \mathcal{R}_{3,j}^p:& &{\mone}\ar@{-}[r] &\\
          \mathcal{R}_{4,j}^p:& &{\mone} &
         }
\]
In this situation, we have the following:
\begin{enumerate}
\item If $\sum_p r_2^p > \sum_p r_3^p$,
then $l(T_{\owe_C(-1)}(\alpha)) < l(\alpha)$.
\item If $\sum_p r_2^p < \sum_p r_3^p$,
then $l(T_{\owe_C(-2)}(\alpha)) < l(\alpha)$.
\end{enumerate}
\end{lem}

\begin{proof}
The assumption in (i) and \cite[Lemma 4.15]{IU05}
imply the inequality
$$
\sum_p l( T_{\owe_C(-1)} (H^p(\alpha)|_U))< \sum_p l(H ^p(\alpha)|_U)(=l(\alpha|_U)).
$$ 
Combining this with Lemma \ref{lem:restriction3},
we obtain the inequality \eqref{eqn:l_cohom} for $\Psi=T_{\owe_C(-1)}$.
Then the desired result follows as explained above.

The proof of  (ii) is similar.
\end{proof}

%%%
%%%
%%% lemma:r_4=0

\begin{lem}[cf.~Lemma 6.6 in \cite{IU05}]\label{lemma:r_4=0}
Under the assumptions of Lemma \ref{lemma:fundamental},
assume that  $\sum_p r_2^p = \sum_p r_3^p \ne 0$ holds.
Then $r_4^p=0$ for all $p$.
\end{lem}

\begin{proof}
According to the decomposition of $H^p(\alpha)|_U$
in Lemma \ref{lemma:fundamental}, 
we can decompose $H^p(\alpha)$ as 
$$
H^p(\alpha)=\bigoplus_{j'}\widetilde{\mathcal{R}_{1,j'}^p}
\oplus\bigoplus_j^{r_2^p}\widetilde{\mathcal{R}_{2,j}^p}
\oplus\bigoplus_j^{r_3^p}\widetilde{\mathcal{R}_{3,j}^p}
\oplus\bigoplus_j^{r_4^p}\mathcal{R}_{4,j}^p,
$$
where $\widetilde{\mathcal{R}_{1,j'}^p}$ satisfies $s(\widetilde{\mathcal{R}_{1,j'}^p})\ne s$, and
$\widetilde{\mathcal{R}^p_{k,j}}$ ($k=2,3$) is a sheaf in $\Sigma(\alpha)$ such that $\mc{R}_{k,j}^p$ is a direct summand of $\widetilde{\mc{R}_{k,j}^p}|_U$.
Here note that $\widetilde{\mathcal{R}_{2,j}^p}|_U$ and
$\widetilde{\mathcal{R}_{3,j}^p}|_U$ may possibly
contain several direct summands of the form $\mathcal{R}_{1,j}^p$.

Apply the proof of
\cite[Lemma 6.6]{IU05} for this decomposition,
then we obtain the assertion.
Note that in the proof we use  
\cite[Lemmas 6.4, 6.5]{IU05}.
%, but they hold without any changes in our situation (except replacing the notations $X,Z$ with $S,Z_0$ respectively).
However,  our setting requires some slight changes. 
For example, we should replace the notations $X,Z$ with $S,Z_0$ respectively, and 
the vertical arrows of the diagram \cite[(6.3)]{IU05} are not isomorphism any more, but injective. 
Hence, we define $\bar{\eta}^p$ to be the following composite:
\begin{equation*}
\xymatrix{
%1st column
{\owe_C(-1)^{\oplus r_2^p}} \ar[rrr]^{\bar\eta^p} \ar@{^{(}->}[d]
&&&
{\owe_C(-1)^{\oplus r_3^{p-1}}[2]} 
\\
%2nd column
{\mc{H}om_X(\owe_C, \bigoplus_j^{r_2^p}\widetilde{\mathcal{R}_{2,j}^p})} \ar@{^{(}->}[r]
&
{\bigoplus_j^{r_2^p}\widetilde{\mathcal{R}_{2,j}^p}} \ar[r]^{\eta^p}
&
\bigoplus_j^{r_3^{p-1}}\widetilde{\mathcal{R}_{3,j}^{p-1}}[2] \ar@{->>}[r]
&
(\bigoplus_j^{r_3^{p-1}}\widetilde{\mathcal{R}_{3,j}^{p-1}})|_C[2]\ar@{->>}[u]
}
\end{equation*}
Here the right vertical arrow is the projection to the direct summand.
\end{proof}

%%%%
%%%% 
%%%% lemmaA

\begin{lemmaA}[cf.~Lemma A in \cite{IU05}]\label{lemmaA}
Let $\alpha \in D_{Z_0}(S)$, $U=U_s$, and $C=C_s$ be as above.
Assume that we can write
$$
\bigoplus_p\mathcal{H}^p(\alpha)|_{U}=
\bigoplus_j^{r_1}\mathcal{R}_{1,j}
\oplus\bigoplus_j^{r_2}\mathcal{R}_{2,j}
\oplus\bigoplus_j^{r_3}\mathcal{R}_{3,j}
\oplus\bigoplus_j^{r_4}\mathcal{R}_{4,j},
%\oplus \mc{S}
$$
such that $\mathcal{R}_{k,j}\in\Sigma(\alpha)_U$, and they  are of the the following form:
%%%%%%

%%%%%%%
\[ 
\xymatrix@R=1ex@M=.3ex{     &C_{s-1} & C_s & C_{s+1} \\
          \mathcal{R}_{1,j}:& \ar@{-}[r]&{\no}\ar@{-}[r]&\\
          \mathcal{R}_{2,j}:& &{\no}\ar@{-}[r] &\\
          \mathcal{R}_{3,j}:& &{\no} &  \\ 
          \mathcal{R}_{4,j}:&  \ar@{-}[r]&{\no} &    
          }
\]
%and where $\Supp \mc{S} \cap C = \emptyset$. 
Suppose that either $r_3\ne 0$ or $r_2\cdot r_4\ne 0$ holds, 
and suppose furthermore that $\Supp \alpha\ne C$.
Then, there is an integer $a$ such that $l(T_{\mc{O}_{C}(a)} (\alpha))<l(\alpha)$.
\end{lemmaA}

\begin{proof}
The proof goes parallel to that of \cite[Lemma A]{IU05}.
%On \cite[Lemma 6.1]{IU05}, just use  Proposition \ref{prop:BBDG} instead of it. 
Let us denote by $\widetilde{\mc{R}_{1,j}}$ an element in 
$\Sigma(\alpha)$ such that $\mc{R}_{1,j}$ is a direct summand of $\widetilde{\mc{R}_{1,j}}|_U$.

When $r_2=r_4=0$, we can see that
$$
\chi(\widetilde{\mc{R}_{1,j}},\mc{R}_{3,k})=0
$$
for any $j,k$. Note that \cite[Lemma 6.2]{IU05} is true without any changes in our situation
so that we can conclude $r_1=0$. This contradicts the asumption that $\Supp \alpha\ne C$. 
%Then we deduce a contradiction as in the proof of \cite[Lemma A]{IU05}.
 Therefore, by the symmetry, 
we may safely assume that $r_2\ne 0$.

In the case $r_2\cdot r_4 \ne 0$, we can see from 
Lemma \ref{lem:degree_restriction} that there is an integer $a$
such that
$$
\deg_C\mc{R}_{2,h}=\deg_C\mc{R}_{4,k}=a
$$
for all $h,k$,
and that $\deg_C \mc{R}_{3,j}$ is $a$ or $a-1$.
%Then by Lemma \ref{lem:restriction3} (ii),
Then by Lemma \ref{lem:restriction3} (ii) and \cite[Lemma 4.15]{IU05},
the inequality \eqref{eqn:l_cohom}
holds for $\Psi=T_{\mc{O}_C(a-1)}$. %(See the argument in the proof of Lemma \ref{lemma:fundamental}.)
Hence, we obtain the desired result. 

What remains is the case
$r_2\cdot r_3\ne 0$ and $r_4=0$.
First, suppose  that 
$$
\deg _{C_s}\mc{R}_{3,h}=\max_{k,j}\deg_{C_s}\mc{R}_{k,j}(=:b)
$$ 
holds for some $h$.
Then 
$\deg _{C_s}\mc{R}_{1,j}=\deg _{C_s}\mc{R}_{2,j}=b$ for all $j$, 
and if we put  $\Psi= T_{\mc{O}_C(b-1)}$, 
again using Lemma \ref{lem:restriction3} (ii) along with \cite[Lemma 4.15]{IU05}, 
we obtain the inequality \eqref{eqn:l_cohom}.

Next, suppose that
$\deg _{C_s}\mc{R}_{3,j}=b-1$ holds for all $j$.
In this case, just apply Lemmas \ref{lemma:fundamental} and \ref{lemma:r_4=0}, which imply the assertion.
\end{proof}

%%%
%%%
%%% lemmaB

\begin{lemmaB}[cf.~Lemma B in \cite{IU05}]\label{lemmaB}
Let $\alpha\in D_{Z_0}(S)$ be a spherical object and fix positive integers $s,t$ with $1\le t-s\le n-2$. 
Take a sufficiently small complex analytic neighbourhood $U=U_{s,\ldots,t}$ 
of $C_s\cup\cdots\cup C_t$ and 
assume that we can write
$$
\bigoplus_p H^p(\alpha)|_{U}
=\bigoplus_j^{r_1}\mathcal{R}_{1,j}\oplus\bigoplus_j^{r_2}\mathcal{R}_{2,j}
\oplus\bigoplus_j^{r_3}\mathcal{R}_{3,j}
\oplus\bigoplus_j^{r_4}\mathcal{R}_{4,j},%\oplus\mathcal{S},
$$
where $\mathcal{R}_{k,j}\in\Sigma(\alpha)_U$, and they are of the forms 
\[ 
\xymatrix@R=1ex@M=.3ex{     &    C_{s-1}& C_s & C_{s+1} & & C_{t-1} & C_t & C_{t+1}\\
          \mathcal{R}_{1,j}:&\ar@{-}[r]&{\no}\ar@{-}[r]&{\no}\ar@{-}[r]&\cdots \ar@{-}[r]&{\no}\ar@{-}[r]&{\no}\ar@{-}[r]& \\
          \mathcal{R}_{2,j}:&          &{\no}\ar@{-}[r]&{\no}\ar@{-}[r]&\cdots \ar@{-}[r]&{\no}\ar@{-}[r]&{\no}\ar@{-}[r]& \\
          \mathcal{R}_{3,j}:&          &{\no}\ar@{-}[r]&{\no}\ar@{-}[r]&\cdots \ar@{-}[r]&{\no}\ar@{-}[r]&{\no}& \\
          \mathcal{R}_{4,j}:&\ar@{-}[r]&{\no}\ar@{-}[r]&{\no}\ar@{-}[r]&\cdots \ar@{-}[r]&{\no}\ar@{-}[r]&{\no}& }.                   
\]
%and where $\Supp\mathcal{S}\cap (C_s\cup\cdots\cup C_t)=\emptyset$.
Suppose that either $r_3\ne 0$ or $r_2\cdot r_4\ne 0$ holds.
Then there is  
$$
\Phi\in  \Span{T_{\owe_{C_l}(a)}\bigm|  a\in \Z, s\le l\le t }
$$
such that $ l(\Phi(\alpha))<l(\alpha)$.
\end{lemmaB}

\begin{proof}
The proof goes parallel to that of \cite[Lemma B]{IU05}.
Note that the straight-forward analogues of
\cite[Lemmas 6.7, 6.8]{IU05} hold for
$H^p(\alpha)|_U$, and use them.
%and assuming that integers $s,t$ satisfy $1\le t-s\le n-2$. 
\end{proof}

%%%%%%%%%%%%%%%%%%%%%%%%%%%%%%%%%%%%%%%%%%%%%%%%%%%%%%%%%%

\subsection{The proof of Proposition \ref{proposition:step -1 of A_n}}\label{subsec:step -1 of A_n}

Notice that if we show the existence of an autoequivalence $\Phi\in B_0$
such that $l(\alpha)>l(\Phi(\alpha))$, under the assumption that $l(\alpha) >1$, 
then we can prove the statement by induction on $l(\alpha)$.
Recall that the proof is already done in \cite[Proposition 1.7]{IU05} 
in the case $\Supp\alpha\ne Z_0$, since in this case $\alpha$ is supported by a 
chain of projective lines, contained in $Z_0$.
Hence, we may assume $\Supp\alpha=Z_0=C_1\cup\cdots\cup C_n$.
 
For $\mc{S}=\mc{S}_{s}(a_{s},\ldots , a_{t})\in \Sigma(\alpha)$, define an integer 
$s(\mathcal{S})$ 
by the properties 
$$
C_{s(\mathcal{S})}=C_s \mbox{ and } 1\le  s(\mathcal{S})\le n,
$$
and an integer $t(\mathcal{S})$ 
by the properties 
$$
C_{t(\mathcal{S})}=C_t \mbox{ and } 
s(\mc{S})\le  t(\mathcal{S})\le s(\mc{S})+n-2.
$$
Here, note that Lemma 
\ref{lem:degree_restriction} (ii) guarantees that,
 for $\mc{R} \in \Sigma(\alpha)$, there are no elements $\mc{S}\in\Sigma (\alpha)$ such that $C_{t(\mc{S})}=C_{s(\mc{R})-1}$ 
or $C_{t(\mc{R})}=C_{s(\mc{S})-1}$.  Thus, we have
$$
l_{s(\mc{R})-1}(\alpha) < l_{s(\mc{R})}(\alpha)
\text{ and } l_{t(\mc{R})}(\alpha) > l_{t(\mc{R})+1}(\alpha).
$$
Therefore, we can find integers $s_0$ and $t_0$ such that 
$$l_{s_0-1}(\alpha) < l_{s_0}(\alpha)=l_{s_0+1}(\alpha)=\dots=l_{t_0}(\alpha)>l_{t_0+1}(\alpha).$$
Then we are in the situation of Lemma A (if $s_0=t_0$) or Lemma B (if $s_0<t_0$). So the proof is done.

%%%
%%%
%%% remark

\begin{rem}\label{remark:R_0}
Take an arbitrary element $\mc{R} \in \Sigma(\alpha)$.
Then, in the proof above, we can 
find $s_0, t_0$ such that $s(\mc{R}) \le s_0 \le t_0 \le t(\mc{R})$.
Thus, Lemma A or B provides
$$
\Phi\in\Span{T_{\mc{O}_{C_l}(a)} \bigm|  a\in\Z, s(\mc{R}) \le l \le t(\mc{R}) }
$$
such that $l(\alpha)>l(\Phi(\alpha))$.
%%%%%%%%%%%%%%%%%%%
%In the proof above, note that we can 
%find 
%$$
%\Phi\in\Span{T_{\mc{O}_{C_l}(a)} \bigm|  a\in\ZZ, C_l\subset\Supp\mc{R}_0 }
%$$
%such that $l(\alpha)>l(\Phi(\alpha))$. 
%If we exchange $C_l$ and $C_{n-l+1}$ in the proof, 
%we obtain $\mc{R}'_0\in\Sigma(\alpha)$ (which may be different from $\mc{R}_0$)
%and 
%$$
%\Phi '\in\Span{T_{\mc{O}_{C_l}(a)} \bigm|  a\in\ZZ, C_l\subset\Supp\mc{R}'_0 }
%$$ 
%such that $l(\alpha)>l(\Phi '(\alpha))$.
%%%%%%%%%%%%%%%%%%%%
We shall use this remark in \S \ref{sec:-2}.
\end{rem}

%%%%%%%%%%%%%%%%%%%%%%%%%%%%%%%%%%%%%%%%%%%%%%%%%%%%%%%%%%
%%%%%%%%%%%%%%%%%%%%%%%%%%%%%%%%%%%%%%%%%%%%%%%%%%%%%%%%%%
%%%%%%%%%%%%%%%%%%%%%%%%%%%%%%%%%%%%%%%%%%%%%%%%%%%%%%%%%%  

\section{The proof of Proposition \ref{proposition:step -2 of A_n}}\label{sec:-2}

%%%%%%%%%%%%%%%%%%%%%%%%%%%%%%%%%%%%%%%%%%%%%%%%%%%%%%%%%%
%In this section we prove Proposition \ref{proposition:step -2 of A_n}.
\subsection{Facts needed for the proof of Proposition \ref{proposition:step -2 of A_n}}
Let $\Phi$ be an autoequivalence of $D_{Z_0}(S)$ that preserves the cohomology class of a point.
Put $\alpha=\Phi(\owe_{C_1})$ and $\beta=\Phi(\owe_{C_1}(-1))$.
Applying Proposition \ref{proposition:step -1 of A_n} and  the shift functor $[1]$, 
we may assume that %$l(\alpha)=1$, and in particular, 
$$
\alpha=\mc O_{C_b}(a)
$$ 
for some $a,b\in \Z$ with $1\le b\le n$.
To prove 
Proposition \ref{proposition:step -2 of A_n}, it suffices to show the following;

\begin{cla}\label{cla:l(beta),l(alpha)}
Suppose that $l(\beta)>1$. There is an autoequivalence 
$\Psi\in B_0$ 
such that $l(\Psi(\alpha))=1$ and $l(\beta)>l(\Psi(\beta))$.
\end{cla}

In fact, Proposition \ref{proposition:step -2 of A_n} easily follows from this:\\

\noindent
{\it Proof of Proposition \ref{proposition:step -2 of A_n}.}
By Claim \ref{cla:l(beta),l(alpha)}, we can reduce the problem to the case $l(\alpha)=l(\beta)=1$.
In this case, the supports of $\alpha$ and $\beta$ must be the same by Condition \ref{condition:(4)} below.
Therefore, we get the conclusion from the $A_1$-case \cite[Proposition 1.6]{IU05},
and we can complete the proof of 
Proposition \ref{proposition:step -2 of A_n} by induction on $l(\beta)$.
\qed\\

The most of the proof of Claim \ref{cla:l(beta),l(alpha)} goes parallel to that of \cite[Claim 7.1]{IU05}.

\begin{fact}[cf.~Condition 7.2 in \cite{IU05}]\label{condition:(1)}
We may assume 
$$\max \{\deg_{C_b}\mc{R}\,|\,\mc{R}\in\Sigma(\beta)_{U_b}, \Supp \mc{R} \supset C_b \}=0.$$
Especially, $\deg_{C_b}\mc{R}=0$ or $-1$ 
for all $\mc{R}\in\Sigma(\beta)_{U_b}$ with $\Supp \mc{R} \supset C_b$ by Lemma \ref{lem:degree_restriction} (iii-1).
\end{fact}

The relations between $\owe_{C_1}$ and $\owe_{C_1}(-1)$ impose conditions on $a$ and $\beta$.
From the spectral sequence %(\ref{equation:spectral1})
\begin{equation}\label{equation:spectral}
E_2^{p, q}=\Ext_S^p(H^{-q}(\beta), \mc{O}_{C_b}(a)) \Longrightarrow
\Hom_{D(S)}^{p+q}(\beta, \alpha)= \begin{cases}
                                             \C^2 & p+q=0 \\
                                             0 & p+q\ne 0,
                                             \end{cases}
\end{equation}
we obtain
% Condition (2)
\begin{fact}[cf.~Condition 7.3 in \cite{IU05}]\label{condition:(2)}
$E_2^{1, q}=0$ for $q\ne -1$
\end{fact}
\noindent
and
% Condition(3)
\begin{fact}[cf.~Condition 7.4 in \cite{IU05}]\label{condition:(3)}
$d_2^{0,-1}:E_2^{0,-1}\to E_2^{2,-2}$ is injective, 
$d_2^{0, 0}:E_2^{0, 0}\to E_2^{2,-1}$ is surjective,
and $d_2^{0, q}:E_2^{0, q}\to E_2^{2,q+1}$ are isomorphisms for all $q\ne 0,-1$.
\end{fact}

\noindent
In addition to Facts \ref{condition:(2)} and \ref{condition:(3)}, \eqref{equation:spectral} implies
\begin{equation}\label{equation:addCondition(3)}
\dim \Coker d_2^{0, -1}+\dim \Ker d_2^{0,0}+\dim E_2^{1, -1} =2.
\end{equation}

%But \cite[Condition 7.5]{IU05} should be replaced with the following:

To show the following, we need the assumption that 
$\Phi$ preserves the cohomology class $\ch (\mc{O}_x)$ in 
$H^4(S,\Q)$.\footnote{If we drop this assumption, 
we can just conclude $c_1(\beta)=[C_b]+k[Z_0]$ for some $k\in \Z$ by a similar proof to that of Condition 7.5 in \cite{IU05}.}

%%%
%%%
%%%

\begin{fact}[cf.~Condition 7.5 in \cite{IU05}]\label{condition:(4)}
The equality $c_1(\beta)=[C_b]$ holds in the cohomology group $H^2(S,\Q)$.
\end{fact}

\begin{proof}
\if0
The radical of the Euler form
$\chi\colon K_Z(S)\times K_Z(S) \to \Z$ is 
$$
\Z[\owe _x]\oplus \Z [\owe _Z].
$$
(In $A_n$-case in \cite{IU05}, this is just $\Z[\owe _x]$.)
Then by a similar proof in \cite[Condition 7.5]{IU05} 
 we obtain $c_1(\beta)=C_b+k Z$ for some $k\in \Z$.
\fi
By the choice of $\Phi$, we have
$$
0=c_1(\mc O_x)=c_1(\Phi(\mc{O}_x))=c_1(\alpha)-c_1(\beta),
$$
which gives the assertion.
\end{proof}

%%%
%%%
%%% claim:a>=-1

\begin{cla}[cf.~Claim 7.6 in \cite{IU05}]\label{claim:a>=-1}
We have $a\ge -1$.
\end{cla}
\begin{proof}
The proof is similar to that of 
 \cite[Claim 7.6]{IU05}. 
\end{proof}
\if0
\begin{cla}\label{claim:usuful0}
Fix $q \ne 1$.
Then we have $\deg_{C_b}\mc{R}>a$ for all 
direct summands $\mc{R} \in \Sigma(\beta)_{U_b}$ of $H^q(\beta)|_{U_{b}}$ with 
$$\Supp\mc{R}\cap (C_{b-1}\cup C_b\cup C_{b+1})=\Supp\mc{R}.$$
If, in addition, we suppose that $a\ge 0$, 
 there are no such direct summands.
\end{cla}
\begin{proof}
Fact \ref{condition:(2)} and 
Lemma \ref{lem:degree_restriction} yields the equality $a<\deg_{C_b}\mc{R}_i$.

Fact \ref{condition:(1)} implies that  
$\deg_{C_b}\mc{R}_i=-1$ or $0$, hence we obtain the second assertion.
\end{proof}
\fi

\begin{cla}[cf.~Claim 7.7 in \cite{IU05}]\label{claim:useful}
Fix $q \ne 0$.
If $E^{2,-q-1}_2= 0$ in \eqref{equation:spectral}, then we have $\deg_{C_b}\mc{R}>a$ for all 
direct summands $\mc{R} \in \Sigma(\beta)_{U_b}$ of $H^q(\beta)|_{U_b}$ with 
$\Supp \mc{R} \supset C_b$.
If, in addition, we suppose that $a\ge 0$, then we get $C_b\not\subset\Supp H^q(\beta)$.
\end{cla}
\begin{proof}
The proof of Claim \ref{claim:useful} is similar to that of 
 \cite[Claim 7.7]{IU05}. 
\end{proof}

%The following remark is useful.
 
\begin{rem}\label{rem:nonempty}
Below we use the notation $t(\mc{R})$ and
$s(\mc{R})$ for $\mc{R}\in\Sigma (\beta)$ defined in 
\S \ref{subsec:step -1 of A_n}, and recall 
that $b$ satisfies $1\le b\le n$ by the definition. 

If there is an element $\mc{R}\in\Sigma (\beta)$ with either 
$$
t(\mc{R})-n+1<b<s(\mc{R})-1 
\text{ or }
 t(\mc{R})+1<b,
$$
then we can find 
$\Psi\in \Span{T_{\mc{O}_{C_l}(a)} \bigm| a\in\Z, C_l \subset \Supp \mc{R} }$
such that $\Psi(\alpha)\cong \alpha$ and 
$l(\Psi(\beta))< l(\beta)$
by Remark \ref{remark:R_0}.
Therefore, we may assume that
$$
s(\mc{R})-1 \le b \le t(\mc{R})+1 \text{ or } 
b+n \le t(\mc{R})+1
$$
for all $\mc{R}\in\Sigma (\beta)$.
%\end{condition}
\end{rem}

Now we divide the proof into cases as in \cite[Division into cases on page 426]{IU05}.
We have only to consider the three cases:

\begin{division}
\begin{enumerate}
\item[(i)] $C_b\subset\Supp\mc{R}$ for all $\mc{R}\in\Sigma(\beta)_{U_b}$,
\item[(ii)] 
there is an $\mc{R} \in \Sigma(\beta)_{U_b}$ with $\Supp \mc{R} \cap C_b = C_{b+1} \cap C_b$,
but there is no $\mc{R}' \in \Sigma(\beta)_{U_b}$ with $\Supp \mc{R}' \cap C_b = C_{b-1} \cap C_b$,
\item[(iii)]
there are $\mc{R}, \mc{R}' \in \Sigma(\beta)_{U_b}$ with 
$\Supp \mc{R} \cap C_b = C_{b+1} \cap C_b$
and $\Supp \mc{R}' \cap C_b = C_{b-1} \cap C_b$.
\end{enumerate}
\end{division}
%Similarly in the proof of \cite[Condition 7.4]{IU05}, 
%we should often restrict a sheaf on $Z$ to $U$ for an appropriate 
%open neiborhood of several $(-2)$-curves. 

%%%%%%%%%%%%%%%%%%%%%%%%%%%%%%%%%%%%%%%%%%%%%%%%%%%%%%%%%%%%%%%%%%%%%
%%%%%%%%%%%%%%%%%%%%%% Case (i) %%%%%%%%%%%%%%%%%%%%%%%%%%%%%%%%%%%%%
%
\subsection{Case (i)}
In this case, we can find $\Psi$ in Claim \ref{cla:l(beta),l(alpha)} in a similar way to 
that of \cite[\S7.3. Case (i)]{IU05},
after some obvious changes; for instance
 \cite[Claim 7.8]{IU05} should be replaced by
\begin{cla}
$$
\mc{O}_{\cdots\cup C_b}(*,-1)|_{U_{b}},\mc{O}_{C_b\cup\cdots}(-1,*)|_{U_{b}}, \mc{O}_{\cdots\cup C_b\cup\cdots}(*,-1,*)|_{U_{b}}
\not\in\Sigma(\beta)_{U_{b}}.
$$
\end{cla}

%\paragraph{Case (i.1): $a\ge 1$.}
%\paragraph{Case (i.2): $a=0$.}
%\paragraph{Case (i.3): $a=-1$.}

%%%%%%%%%%%%%%%%%%%%%%%%%%%%%%%%%%%%%%%%%%%%%%%%%%%%%%%%%%%%%%%%%%%%%%%
%%%%%%%%%%%%%%%%%%%%%% Case (ii) %%%%%%%%%%%%%%%%%%%%%%%%%%%%%%%%%%%%%

%
\subsection{Case (ii)}

The existence of $\mc{R} \in \Sigma(\beta)_{U_b}$ with $\Supp \mc{R} \cap C_b=C_b \cap C_{b+1}$
and Lemma \ref{lem:degree_restriction} (ii) imply the non-existence of $\mc{S} \in \Sigma(\beta)_{U_b}$ with
$\Supp \mc{S} \cap C_{b+1} = C_b\cap C_{b+1}$.
Note that $n> 2$ in this case by Lemma \ref{lem:degree_restriction_n=2} (i).
Furthermore, we have
\begin{align*}
\Sigma(\beta)_{U_{b}}\subset&\bigl\{ \mc{O}_{C_b\cup\cdots}(a',*)|_{U_{b}}, \mc{O}_{C_{b+1}}(*)|_{U_{b}},
\mc{O}_{\cdots\cup C_b\cup\cdots}(*,a',*)|_{U_{b}}\bigm|a'=-1,0 \}.
\end{align*}
Then we can find $\Psi$ in Claim \ref{cla:l(beta),l(alpha)} in a similar way to that of \cite[Case (ii)]{IU05}.
\if0
Here again we need some minor changes; 
for instance in the case $a=0$, 
we can see the equality 
$$
(H^0(\beta),H^1(\beta))=
(\mc{O}_{C_b\cup\cdots\cup C_{b''+kn}}(0,*),
\mc{O}_{C_{b+1}\cup\cdots\cup C_{b''+kn}}(*))
$$
holds for some $k\in \Z_{\ge 0}$ and 
$b+1\le b''\le b+n-2$,
instead of the equality in \cite[page 428, line 6]{IU05}.
\fi
%%%%%%%%%%%%%%%%%%%%%%%%%%%%%%%%%%%%%%%%%%%%%%%%%%%%%%%%%%%%%%%%%%%%%%%
%%%%%%%%%%%%%%%%%%%%%% Case (iii) %%%%%%%%%%%%%%%%%%%%%%%%%%%%%%%%%%%%%

\subsection{Case (iii)}

Fact \ref{condition:(2)} implies that $\mc{R}$ 
and $\mc{R}'$ above must be in $H^1(\beta)$.
Moreover, 
they are unique in a decomposition of $H^1(\beta)$,
by virtue of the inequality $\dim E_2^{1, -1} \le 2$ from \eqref{equation:addCondition(3)}.

We can exclude the case $n=2$ as follows. Suppose that $n=2$. Recall that 
$\mc{R}$ is the unique element in $\Sigma(\beta)$ such that $C_{s(\mc {R})}\ne C_{b}$, 
and $\mc{R}'$ is the unique element in $\Sigma(\beta)$ such that $C_{t(\mc {R}')}\ne C_{b}$. 
It follows from 
Lemma \ref{lem:degree_restriction_n=2} (i) 
that every element $\mc{S}\in \Sigma (\beta)$ satisfies that $C_{s(\mc {S})}=C_{t(\mc{S})}$.
Therefore 
% together with Fact \ref{condition:(2)}  and the equation \eqref{equation:addCondition(3)} that
$\beta[-1]$ is an indecomposable sheaf $\mc {S}(=\mc{R}=\mc{R}')$ satisfying $C_{s(\mc {S})}=C_{t(\mc{S})}\ne C_b$. 
In this case, Fact \ref{condition:(4)} cannot be satisfied.

Now, Lemma \ref{lem:degree_restriction} allows us to write
$$
\bigoplus_p H^p(\beta)|_{U_{b-1,b}}=
\bigoplus_j^{r_1}\mathcal{R}_{1,j}\oplus\bigoplus_j^{r_2}\mathcal{R}_{2,j}
\oplus\mathcal{R}_3\oplus\mathcal{R}_4,
$$
where $\mathcal{R}_{k,j}$'s, $\mathcal{R}_3$ and $\mathcal{R}_4$ are sheaves in $\Sigma(\beta)_{U_{b-1,b}}$ 
of the following forms:
\[ 
\xymatrix@R=1ex@M=.3ex{     &               & C_{b-1} & C_b &  C_{b+1}& \\
          \mathcal{R}_{1,j}: & \ar@{-}[r]&{\no}\ar@{-}[r]&{\no}\ar@{-}[r]&{\no}\ar@{--}[r]&   &\\         
          \mathcal{R}_{2,j}: &           &{\no}\ar@{-}[r]&{\no}\ar@{-}[r]&{\no}\ar@{--}[r]&   &\\
          \mathcal{R}_3    : & \ar@{--}[r]&{\mone}          &                 &{\no}\ar@{--}[r]& & :\mathcal{R}_4\\
          \alpha:           &          &                  &{\cia}          &                 &}
\]
Here, we assume that $\deg_{C_{b-1}}\mc{R}_3=-1$ for simplicity.

%%%%%%%%%
\begin{cla}
We have $a=-1$.
\end{cla}
\begin{proof}
A similar proof to that of \cite[Claim 7.11]{IU05} works.
\end{proof}

%%%%
%%%%
%%%%
%\paragraph{Case (iii.2): $a=-1$.}

The inequality $\dim E_2^{1, -1} \le 2$ from \eqref{equation:addCondition(3)}
also implies that 
$$
\Ext_S^1(\mathcal{R}_{k,j},\mc{O}_{C_b}(-1))=0
$$
for $k=1, 2$ and for all $j$. 
In particular, we get
$$
\deg_{C_b}\mc{R}_{1,j}=\deg_{C_b}\mc{R}_{2,j}=0.
$$

Recall that there is a unique sheaf $\mc{R}\in \Sigma(\beta)$ satisfying $C_{t(\mc{R})}=C_{b-1}$,
and $\mc{R}|_{U_{b-1,b}}$ contains $\mc{R}_3$ as a direct summand.
Now, we give a proof for Case (iii) by induction on $l(\mc{R})$.

First, suppose $l(\mc{R})=1$. In this case $\mc{R}=\mc{R}_3$, and 
we write
$$
\bigoplus_j^{r_2}\mathcal{R}_{2,j}=\bigoplus_j^{s_1}\mathcal{S}_{1,j}
\oplus\bigoplus_j^{s_2}\mathcal{S}_{2,j},
$$
where $\mathcal{S}_{k,j}$'s are sheaves in $\Sigma(\beta)_{U_{b-1,b}}$ of the following forms.
\[ 
\xymatrix@R=1ex@M=.3ex{ &               & C_{b-1} & C_b &  C_{b+1}& \\
          \mathcal{R}_{1,j}:& \ar@{-}[r]&{\zero}\ar@{-}[r]&{\zero}\ar@{-}[r]&{\no}\ar@{--}[r]& \\
          \mathcal{S}_{1,j}:&          &{\zero}\ar@{-}[r] &{\zero}\ar@{-}[r]&{\no}\ar@{--}[r]& \\
          \mathcal{S}_{2,j}:&          &{\mone}\ar@{-}[r]&{\zero}\ar@{-}[r]&{\no}\ar@{--}[r]& \\
          \mathcal{R}_3:&              &{\mone}         &                 &{\no}\ar@{--}[r]& &:\mathcal{R}_4& \\
          \alpha:           &          &                  &{\mone}          &                 &}
\]

Note that we may assume $r_1\ne 0$ since otherwise $\Supp \beta$ is strictly smaller than $Z_0$, 
in which case \cite[Claim 7.1]{IU05} gives the result.
Because of the existence of $\mc{R}_3$, 
we have $r_2\ne 0$ by \cite[Lemma 6.2]{IU05}, which continues to hold in our setting. Hence,
$s_1\ne s_2$ by Lemma \ref{lemma:r_4=0}.
Define
$$
\Psi _0= \begin{cases}
T_{\mc{O}_{C_{b-1}\cup C_b}(-1,-1)} & \text{ if } s_1<s_2, \\
T'_{\mc{O}_{C_{b-1}\cup C_b}} & \text{ if } s_2<s_1.
\end{cases}
$$
Then $(\Psi_0(\alpha), \Psi _0(\beta))$ fits in Case (ii) and
$\Psi_0(\beta)$ satisfies $l(\Psi_0 (\beta))\le l(\beta)$.
Since we have proved Case (ii), this finishes the case $l(\mc{R})=1$.

Next, suppose $l(\mc{R})>1$.
In this case, Lemma \ref{lem:degree_restriction} (iii-4) implies
$$
\deg_{C_{b-1}} \mc{R}_{2,j}=-1.
$$
Set
$
\Psi'= T_{\owe_{C_b}(-1)} \circ T_{\owe_{C_{b-1}}(-2)}.
$
Then we have 
$$
\Psi'(\alpha) \cong \owe_{C_{b-1}}(-2) \mbox{ and } l(\Psi'(\beta)) \le l(\beta).
$$
%and $(\Psi'(\alpha), \Psi '(\beta))$ again fits in Case (iii).
%
\[ 
\xymatrix@R=1ex@M=.3ex{ &     &  C_{b-2} & C_{b-1} & C_b &  C_{b+1}& &\\
          \Psi'(\mathcal{R}_{1,j}):& &\ar@{--}[l]{\no}\ar@{-}[r]&{\mone}\ar@{-}[r]&{\mone}\ar@{-}[r]&{\no}\ar@{--}[r]& &\\
          \Psi'(\mathcal{R}_{2,j}):& &         & &&{\no}\ar@{--}[r]& &\\
          \Psi'(\mathcal{R}_3):&          &\ar@{--}[l]{\no} &                &{\mone
          }\ar@{-}[r]  &{\no}\ar@{--}[r]& &:\Psi'(\mathcal{R}_4) \\
     \Psi'(\alpha):           & &         &                  {\cib}&          &                 &&}
\]
First, let us consider the case $C_{s(\mc{R})}=C_{b+1}$, equivalently $\mc{R}=\mc{R}'$. 
Take an element $\widetilde{\mathcal{R}_{2,j}}\in \Sigma (\beta)$ such that
$\widetilde{\mathcal{R}_{2,j}}|_{U_{b-1,b}}$ contains $\mathcal{R}_{2,j}$ as a direct summand.
Then $\Psi'(\widetilde{\mathcal{R}_{2,j}})$ satisfies 
$$
t(\Psi'(\widetilde{\mathcal{R}_{2,j}}))-n+1<b-1<s(\Psi'(\widetilde{\mathcal{R}_{2,j}}))-1 
$$
 or 
$$
t(\Psi'(\widetilde{\mathcal{R}_{2,j}}))+1<b-1.
$$
Hence, we can find $\Psi$ as in Claim \ref{cla:l(beta),l(alpha)} 
by Remark \ref{rem:nonempty}.
Therefore, we may assume that $C_{s(\mc{R})}\ne C_{b+1}$.
Note that $\Psi'(H^q(\beta))$ is a sheaf for every $q\in \Z$.
Using the spectral sequence
$$
 E_2^{p,q}= H^p(\Psi'(H^q(\beta))) \Longrightarrow E^{p+q}=H^{p+q}(\Psi'(\beta)),
$$ 
we see $\Psi'(H^q(\beta))=H^q(\Psi'(\beta))$, and 
then $(\Psi'(\alpha), \Psi '(\beta))$ fits in Case (iii).  Furthermore, the assumption 
 $C_{s(\mc{R})}\ne C_{b+1}$ implies that $l(\Psi'(\mc{R}))<l(\mc{R})$.
Hence, we can conclude by induction on  $l(\mc{R})$.

\qed\\

%%%%%%%%%%%%%%%%%%%%%%%%%%%%%%%%%%%%%%%%%%%%%%%%%%%%%%%%%%%%%%%%%%%
%%%%%%%%%%%%%%%%%%%%%%%%%%%%%%%%%%%%%%%%%%%%%%%%%%%%%%%%%%%%%%%%%%%
%%%%%%%%%%%%%%%%%%%%%%%%%%%%%%%%%%%%%%%%%%%%%%%%%%%%%%%%%%%%%%%%%%%

\section{Example}\label{sec:examples}
If $\lambda_S\le 4$, we have $H_S=\{\id_S\}$ and $\Image \Theta=\SL(2,\Z)$ as in Remark \ref{rem:image_Theta} (ii).
For a general elliptic surface $S$, however, it is not easy to describe the group $H_S$ concretely. 
In this section, 
we give an example of elliptic surfaces with $\lambda_S\ge 5$ where
we are able to determine $H_S$, and hence
$\Image \Theta$, more concretely. 

\if0
%%%%%%%%%%%%%%%%%%%%%%%%%%%%%%%%%%%%%%%%%%%%%%%%%%%%%%%%%%%%%%%%%%%
\subsection{Autoequivalences of elliptic ruled surfaces}

Let $S:=\PP(\mc{O}_E\oplus \mc{L})$ be an elliptic ruled surface 
with non-trivial Fourier--Mukai partners, where
$E$ is an elliptic curve, and $\mc{L}$ is in ${}_mE$ for some $m>2$, as in \cite{TU14}. 
We freely use the results in \cite{TU14} below.
Then the group $H_S$ coincides with the group
$$
H_{\hat{E}}^\mc{L}:=\{ k\in(\Z/m\Z)^*\mid \text{ $\exists\psi_1\in \Aut_0 E$ such that $\psi_1^*\mc{L}=\mc{L}^k$} \}.
$$
Note also that 
$\Auteq D(S)=\Auteq ^\dagger D(U)$,
since $\pi$ has no reducible fibers.

Therefore, by Theorem \ref{thm:AeV_extension},
we have the following short exact sequence:
\begin{align*}
1\to 
\Span{\otimes \mathcal{O}_S(D)\mid  D\cdot F=0}
&\rtimes \Aut S \times \Z[2] \to
\Auteq   D(S) \notag\\
&\to
\bigl\{ \begin{pmatrix}
c& a\\
d& b   
\end{pmatrix}\in \Gamma_0(m) \bigm| b\in H_{\hat{E}}^{\mc{L}}                       \bigr\}
\to 1.
\end{align*}
%Here for an integer $b$, coprime with $m$, we again denote by $b$ 
% the corresponding element in $H_{\hat{E}}^{\mc{L}}(\subset (\Z/m\Z)^*)$.

%%%%%%%%%%%%%%%%%%%%%%%%%%%%%%%%%%%%%%%%%%%%%%%%%%%%%%%%%%%%%%%%%%%
\subsection{Autoequivalences of certain rational elliptic surfaces} 
\fi

Let us consider a rational elliptic surface $\pi\colon J\to \PP^1$ with a section, and assume that $\pi$ has four singular fibers of 
types $\mathrm{I}_7$, $\mathrm{I}_2$, $\mathrm{II}$ and
$\mathrm{I}_1$. Such a surface exists by Persson's list \cite{Pe90}.
Take a point $s\in \PP^1$ over which the fiber of $\pi$
is not of type $\mathrm{II}$.
Apply a logarithmic transformation along the point $s$ to obtain a rational elliptic surface $S$ whose Jacobian surface is $J$, 
and $S$ has a multiple fiber of the multiplicity $m$ over the point $s$. Suppose that 
 $m>2$. Then as in \cite[Example 2.6]{Ue11},
 we can show that $H_S=\{\pm 1\}$ (we leave it to the reader to check this). 
%Then Remark \ref{rem:image_delta} implies $\Image\delta =\id_{\PP^1}.$
%
Therefore, Theorem \ref{thm:main} assures that there is a short exact sequence:
\begin{align*}%\label{eqn:conj_Auteq_extension}
1\to \Span{B,\otimes \mathcal{O}_S(D)\mid D\cdot F=0}&\rtimes \Aut S\times  
\Z[2]
\to
\Auteq D(S) \notag\\
& \stackrel{\Theta}\to 
\bigl\{ \begin{pmatrix}
c& a\\
d& b   
\end{pmatrix}\in \Gamma_0(m) \bigm| b\equiv\pm 1\ (\module m) \bigr\}
\to 1.
\end{align*}
In this case, $\Aut S=\Aut_{\PP ^1}S$ is just the semi-direct product of 
the Mordell-Weil group of $S$ and the subgroup of automorphisms preserving the zero section (cf.~\cite[Theorem 1.3.14]{FM94}).
Hence, it can be calculated by using \cite{OS90}.

In the upcoming paper \cite{Ue}, we consider the autoequivalence group and Fourier--Muaki partners of elliptic ruled surfaces. 
%%%%%%%%%%%%%%%%%%%%%%%%%%%%%%%%%%%%%%%%%%%%%%%%%%%%%%%%%%
%%%%%%%%%%%%%%%%%%%%%%%%%%%%%%%%%%%%%%%%%%%%%%%%%%%%%%%%%%
%%%%%%%%%%%%%%%%%%%%%%%%%%%%%%%%%%%%%%%%%%%%%%%%%%%%%%%%%%  

\noindent
Department of Mathematics
and Information Sciences,
Tokyo Metropolitan University,
1-1 Minamiohsawa,
Hachioji-shi,
Tokyo,
192-0397,
Japan 

{\em e-mail address}\ : \  hokuto@tmu.ac.jp


\begin{thebibliography}{BM01}

\bibitem[At57]{At57}
M. Atiyah, Vector bundles over an elliptic curve, Proc. London Math., Soc. 5 (1955), 407--434.

\bibitem[BBDG]{BBDG}
L. Bodnarchuk, I. Burban,  Y. Drozd, G.-M. Greuel, Vector bundles and torsion free sheaves on degenerations of elliptic curves. Global aspects of complex geometry, 83--128, Springer, Berlin, 2006. 
 
\bibitem[BHPV]{BHPV} Barth, Wolf P.; Hulek, Klaus; Peters, Chris A. M.; Van de Ven, Antonius, Compact complex surfaces. Second edition. Ergebnisse der Mathematik und ihrer Grenzgebiete. 3. Folge. A Series of Modern Surveys in Mathematics [Results in Mathematics and Related Areas. 3rd Series. A Series of Modern Surveys in Mathematics], 4. Springer-Verlag, Berlin, 2004. xii+436 pp.

\bibitem[BO95]{BO95}
A.I. Bondal, D.O. Orlov, Semiorthogonal decomposition for algebraic varieties, alg-geom 9712029.



\bibitem[Br98]{Br98}
T. Bridgeland, Fourier--Mukai transforms for elliptic surfaces. J. Reine Angew. Math. 498 (1998), 
115--133. 

\bibitem[Br99]{Br99}
T. Bridgeland, Equivalences of triangulated categories and Fourier--Mukai transforms, 
Bull. London Math. Soc. 31 (1999), 25--34. 

\bibitem[BM98]{BM98}
T. Bridgeland, A. Maciocia, Fourier-Mukai transforms for quotient varieties. math.AG/9811101. 

\bibitem[BM01]{BM01}
T. Bridgeland, A. Maciocia,
Complex surfaces with equivalent derived categories. Math. Z. 236 (2001), 677--697. 

\bibitem[BM02]{BM02}
T. Bridgeland, A. Maciocia, Fourier-Mukai transforms for K3 and elliptic fibrations. J. Algebraic Geom. 11 (2002), no. 4, 629--657. 

\bibitem[BV03]{BV03}
A. Bondal, M. Van den Bergh,
Generators and representability of functors in commutative and noncommutative geometry. 
Mosc. Math. J.  3  (2003),  no. 1, 1--36, 258. 


\bibitem[Fr95]{Fr95}
R. Friedman, Vector bundles and SO(3)-invariants for elliptic surfaces, J. Amer. Math. Soc. 8 (1995), 29--139. 

\bibitem[FM94]{FM94}
R. Friedman, J. W. Morgan, Smooth Four-Manifolds and Complex Surfaces, Springer--Verlag Berlin 
Heidelberg 1994.

\bibitem[Ha66]{Ha66}
R. Hartshorne, 
Residues and duality. 
Lecture notes of a seminar on the work of A. Grothendieck, given at Harvard 1963/64. With an appendix by P. Deligne. Lecture Notes in Mathematics, No. 20 Springer-Verlag, Berlin-New York 1966 vii+423 pp. 

\bibitem[Ha77]{Ha77}
R. Hartshorne, Algebraic Geometry, Springer--Verlag, Berlin Heidelberg New York, 1977.

\bibitem[HLS09]{HLS09}
D. Hern\'andez Ruip\'erez, A.C. L\'opez Mart\'in, F. Sancho de Salas,
Relative integral functors for singular fibrations and singular partners. 
J. Eur. Math. Soc. (JEMS)  11  (2009),  no. 3, 597--625. 


\bibitem[Hu06]{Hu06}
D. Huybrechts, Fourier-Mukai transforms in algebraic geometry. Oxford Mathematical Monographs. The Clarendon Press, Oxford University Press, Oxford, 2006. viii+307 pp. 

\bibitem[IU05]{IU05}
A. Ishii, H. Uehara, Autoequivalences of derived categories on the minimal resolutions of $A_n$-singularities on surfaces. J. Differential Geom. 71 (2005), no. 3, 385--435. 

\bibitem[IUU10]{IUU10} 
A. Ishii, K. Ueda, H. Uehara, Stability conditions on $A_n$-singularities. J. Differential Geom. 84 (2010), no. 1, 87--126. 
  
%\bibitem[Ka12]{Ka12}
%T.  Karayayla, The classification of automorphism groups of rational elliptic surfaces with section. Adv. Math. 230 (2012), no. 1, 1--54


%\bibitem[Ka14]{Ka14}
%T. Karayayla,  Automorphism groups of rational elliptic surfaces with section and constant J-map. 
%Cent. Eur. J. Math. 12 (2014), no. 12, 1772--1795.
  
\bibitem[Ka02]{Ka02}
Y. Kawamata, D-equivalence and K-equivalence. J. Differential Geom. 61 (2002), 147--171.

\bibitem[KO95]{KO95}
S. A. Kuleshov, D. Orlov, 
Exceptional sheaves on Del Pezzo surfaces. 
Russian Acad. Sci. Izv. Math. 44 (1995), no. 3, 479--513
 
\bibitem[LST13]{LST13}
A.C. L\'opez Mart\'in, D. S\'anchez G\'omez, C. Tejero Prieto, Relative Fourier--Mukai functors for Weierstra\ss 
\ fibrations, abelian schemes and fano fibrations, Math. Proc. Cambridge Philos. Soc. 155 (2013) 129--153. 
 
%\bibitem[Mu08]{Mu08}
%D. Mumford, Abelian varieties. With appendices by C. P. Ramanujam and Yuri Manin. Corrected reprint of the second (1974) edition. Tata Institute of Fundamental Research Studies in Mathematics, 5. Published for the Tata Institute of Fundamental Research, Bombay; by Hindustan Book Agency, New Delhi, 2008. xii+263 pp. 
 
\bibitem[OS90]{OS90}
K. Oguiso, T. Shioda, The Mordell--Weil lattice of a rational elliptic surface, Comment. Math. Univ. St. Pauli 40 (1991) 83--99. 

\bibitem[Or97]{Or97}
D. Orlov, Equivalences of derived categories and K3 surfaces. Algebraic geometry, 7. J. Math. Sci. (New York) 84 (1997), 1361--1381.

\bibitem[Or02]{Or02}
D. Orlov, Derived categories of coherent sheaves on abelian varieties and equivalences between them, 
Izv. Math. 66 (2002) 569--594.

\bibitem[Pe90]{Pe90}
U. Persson, Configurations of Kodaira fibers on rational elliptic surfaces, Math. Z. 205 (1) (1990) 1--47.

%\bibitem[Pl07]{Pl07}
%D. Ploog, Equivariant autoequivalences for finite group actions. 
%Adv. Math. 216 (2007), no. 1, 62--74. 

\bibitem[ST01]{ST01}
P. Seidel,  R. Thomas, Braid group actions on derived categories of coherent sheaves, Duke Math. J. 108 (2001), 37--108. 

\bibitem[To03]{To03}
Y. Toda, 
Fourier-Mukai transforms and canonical divisors. 
Compos. Math. 142 (2006), no. 4, 962--982. 

\bibitem[Ue04]{Ue04}
H. Uehara, An example of Fourier-Mukai partners of minimal elliptic surfaces. Math. Res. Lett. 11 (2004), no. 2-3, 371--375.

\bibitem[Ue11]{Ue11}
H. Uehara, A counterexample of the birational Torelli problem via Fourier-Mukai transforms. J. Algebraic Geom. 21 (2012), no. 1, 77--96.

\bibitem[Ue]{Ue}
H. Uehara, Derived categories of elliptic ruled surfaces (preprint).
\end{thebibliography}
\end{document}